\documentclass{amsart}%[review,onefignum,onetabnum]{siamonline190516}
\usepackage{lipsum}
\usepackage{amsfonts}
\usepackage{graphicx}
\usepackage{epstopdf}

\usepackage{comment}

\usepackage{amsmath,amscd,amssymb}
\usepackage{xfrac}
\usepackage{url}
\usepackage{mathtools}
\usepackage{color}
\usepackage{subcaption}
\usepackage[pagebackref,colorlinks,citecolor=blue,linkcolor=magenta, allcolors = black]{hyperref}

\usepackage[capitalize]{cleveref}
\usepackage{tikz}
\usetikzlibrary{shapes,positioning}
\usetikzlibrary{calc}

\usepackage{color}
\usepackage{centernot}
\usepackage{enumitem}
\usepackage[utf8]{inputenc}
\usepackage{algorithm}
\usepackage{algpseudocode}

\newcommand{\GG}{\mathcal{G}}
\newcommand{\HH}{\mathcal{H}}
\newcommand{\DD}{\mathcal{D}}

\newcommand{\RR}{\mathbb{R}}
\newcommand{\spans}{\operatorname{span}}
\renewcommand{\ne}{\operatorname{ne}}

\newcommand{\CIM}{\operatorname{CIM}}
\newcommand{\CIMT}{\operatorname{CIMTree}}
\newcommand{\CGP}{\operatorname{CGP}}

\makeatletter
\DeclareMathAlphabet{\@mymathbb}{U}{bbold}{m}{n}
\newcommand{\zero}{\@mymathbb{0}}
\newcommand{\one}{\@mymathbb{1}}
\makeatother

\DeclareMathOperator{\pa}{pa}

\DeclareMathOperator{\argmax}{argmax}
\DeclareMathOperator{\conv}{conv}

\def\newop#1{\expandafter\def\csname #1\endcsname{\mathop{\rm
#1}\nolimits}}
\newop{STAB}
\newop{conv}
\newop{odd}
\newop{even}
\usepackage{enumitem}
\setlist[enumerate]{leftmargin=.5in}
\setlist[itemize]{leftmargin=.5in}

\newtheorem{theorem}{Theorem}[section]
\newtheorem{proposition}[theorem]{Proposition}
\newtheorem{corollary}[theorem]{Corollary}
\newtheorem{lemma}[theorem]{Lemma}
\newtheorem{example}[theorem]{Example} 
\newtheorem{conjecture}[theorem]{Conjecture}
\newtheorem{question}[theorem]{Question}

\theoremstyle{definition}

\newtheorem{goal}[theorem]{Goal}

\theoremstyle{remark}
\newtheorem{remark}[theorem]{Remark}

\title[Rhombus criterion and the chordal graph polytope]{Rhombus criterion and\\ the chordal graph polytope}

\author{Svante Linusson% \thanks{Matematik, KTH, SE-100 44 Stockholm, Sweden
\and Petter Restadh%\thanks{Matematik, KTH, SE-100 44 Stockholm, Sweden
}

\email[Svante Linusson]{linusson@kth.se}
\email[Petter Restadh]{petterre@kth.se}

\address{Department of Mathematics\\
    KTH Royal Institute of Technology\\
    SE-100 44 Stockholm, Sweden}

\begin{document}

%%%%%%%%%%%%%%%%%%%%%%%%%%
%---ABSTRACT
\begin{abstract}
The purpose of this paper is twofold. We investigate a simple necessary condition, called the rhombus criterion, for two vertices in a polytope not to form an edge and show that in many examples of $0/1$-polytopes it is also sufficient. We explain how also when this is not the case, the criterion can give a good algorithm for determining the edges of high-dimenional polytopes.

In particular we study the Chordal graph polytope, which arises in the theory of causality and is an important example of a characteristic imset polytope. We prove that, asymptotically, for almost all pairs of vertices the rhombus criterion holds. We conjecture it to hold for all pairs of vertices.
\end{abstract}

\maketitle

%%%%%%%%%%%%%%%%%%%%%%%%%%
%---SECTION: Intro
\section{Introduction}

Computing the edge structure of a convex polytope is oftentimes a non-trivial task and the most commonly used algorithms for doing this scales very badly with dimension \cite{EFG16, Zie95}. 
This is especially true for the characteristic imset polytopes (see \cref{sec: cim polytopes}) whose dimension grow exponentially and therefore even small examples are computationally unfeasible. 

Let $P$ be a $d$-dimensional polytope with vertices $V$. The following lemma should be obvious, but can be deduced directly from the general definitions stated in Section \ref{sec: computations}.

\begin{lemma}[Rhombus criterion]
\label{lem: square criterion}
Let $P$ be a polytope with vertex set $V$ and let $\alpha,\ \beta\in V$. 
If there exists vertices $\alpha',\ \beta'\in V\setminus\{\alpha, \beta\}$ such that $\alpha+\beta=\alpha'+\beta'$, then $\conv(\alpha, \beta)$ is not an edge of $P$. 
\end{lemma}

The vertices $\alpha'$ and $\beta'$ in the above lemma will be called  \emph{witnesses of the rhombus criterion}, or simply witnesses, for the pair $\alpha, \beta$. Note that the witnesses can in special situations be on the same two-dimensional face as $\alpha, \beta$, but in general not.
We will say that two vertices $\alpha$, $\beta$ of $P$ fulfill the rhombus criterion if $\conv(\alpha, \beta)$, is either an edge of $P$ or have witnesses. 
Then a polytope $P$ \emph{fulfills the rhombus criterion} if every pair of vertices fulfill the rhombus criterion. 

As we will show later, the rhombus criterion appears in a surprising number of well studied $0/1$-polytopes (see \cref{sec: examples}). 
Primarily it provides us with way to determine whether $\alpha$, $\beta$ is an edge or not significantly easier than the general criterion given in \cref{eq: non-edge}, that requires us to consider all subsets of $V\setminus\{\alpha, \beta\}$ of size $d$. 
In fact, it can be checked in $O(d|V|^2)$ time, which scales linearly in dimension, as opposed to exponentially. 
We will begin with a first non-trivial example and some general theory about the rhombus criterion in \cref{sec: rhombus criterion}.

Polytopes defined by the characteristic imset of graphs has shown important in the study of the theory of causality and independence, see e.g. \cite{SHL10,SV08,S15,CS17}. In \cite{LRS20} the authors prove that all algorithms used for structure learning when finding an underlying directed graph  from conditional independence data, are in fact greedy edge walks on the graph of such a polytope. The rhombus criteria turned out to be very useful in \cite{LRS22} when the edge structure of $\CIM_G$, for $G$ a tree was determined. In \cref{sec: cim polytopes} we define and focus on the characteristic imset polytopes and see why the rhombus criterion is effective in these cases. 
Primarily we show that almost all pairs of vertices of the \emph{chordal graph polytope} fulfill the rhombus criterion, Theorem \ref{thm:split edges}. With this we give a partial answer to a question by Studen\'y, Cussens and Kratochv\'{\i}l, \cite{SCK21}.
We conjecture that the chordal graph polytope fulfills the rhombus criterion, \cref{conj:chordal}, and we raise the question whether it is true for the entire imset polytope $\CIM_n$. 
As we will see in \cref{sec: computations}, if many pairs of vertices have witnesses we can utilize the rhombus criterion to compute the edge structure for some characteristic imset polytopes, not previously known, for example $\CIM_5$.

In \cref{sec: examples} we have gathered further examples of 0/1-polytopes from the literature that does fulfill the rhombus criterion.

%%%%%%%%%%%%%%%%%%%%%%%%%%
%---SECTION: Rhombus criterion
\section{The Rhombus Criterion}
\label{sec: rhombus criterion}

% Let $P$ be a polytope with vertices $V$.
% For easier notation we will say that a pair of vertices $\alpha, \beta\in V$ is an edge of $P$ if $\conv(\alpha, \beta)$ is an edge of $P$. 
% Likewise we will say that $\alpha, \beta\in V$ is a \emph{non-edge} of $P$ if $\alpha, \beta$ is not an edge of $P$.

% \begin{lemma}
% \label{lem: square criterion}
% Let $P$ be a polytope with vertex set $V$ and let $\alpha,\ \beta\in V$. 
% If there exists vertices $\alpha',\ \beta'\in V\setminus\{\alpha, \beta\}$ such that $\alpha+\beta=\alpha'+\beta'$, then $\conv(\alpha, \beta)$ is not an edge of $P$. 
% \end{lemma}

% The vertices $\alpha'$ and $\beta'$ in the above lemma will be called   \emph{witnesses of the rhombus criterion}, or simply witnesses, for the pair $\alpha, \beta$.
% We will say that two vertices $\alpha$, $\beta$ of $P$ fulfill the rhombus criterion if $\alpha$, $\beta$, is either an edge of $P$ or have witnesses. 
% Then a polytope $P$ \emph{fulfills the rhombus criterion} if every pair of vertices either have witnesses or is an edge of $P$.

% %---THEOREM: Edge of Face
% \begin{theorem}
% \label{thm: edge of face}
% Let $P$ be a polytope with vertices $V$, let $\alpha,\ \beta\in V$, and assume $F$ is a face of $P$ containing $\alpha$ and $\beta$. 
% Then the pair $\alpha, \beta$ is an edge of $F$ if and only if it is an edge of $P$. 
% \end{theorem}
For easier notation we will say that a pair of vertices $\alpha, \beta\in V$ is an edge of $P$ if $\conv(\alpha, \beta)$ is an edge of $P$. 
Likewise we will say that $\alpha, \beta\in V$ is a \emph{non-edge} of $P$ if $\alpha, \beta$ is not an edge of $P$.
Let us, as an example, consider the spanning tree polytope. 
Let $[n]\coloneqq\{1, \dots, n\}$. 
Given a tree $T=([n], E)$ we define 
\[
v_T=\sum_{ij\in E}e_{ij}
\]
where $e_{ij}$ is the standard basis of $\mathbb{R}^{\binom{[n]}{2}}$. 
Then the \emph{spanning tree polytope} $T_n$ is defined as the convex hull of all $v_T$, that is
\[
T_n\coloneqq \conv\left(v_T\colon T=([n], E)\text{ is a tree}\right).
\]
The polytope $T_n$ is known to be $\binom{n}{2}-1$ dimensional, we have $\langle\mathbf{1}, v_T\rangle=n-1$ for all trees $T$. 
From the work of \cite{GGPS87} we have the following classification of the edges of $T_n$. 

\begin{theorem}\cite{GGPS87}
\label{thm: edges of spanning tree polytope}
For two trees $T_1$ and $T_2$, we have that $v_{T_1}$, $v_{T_2}$ is an edge of $T_n$ if and only if $v_{T_1}-v_{T_2} = e_{ij}-e_{i'j'}$. 
\end{theorem}
Equivalently, if we add in the edge $ij$ to $T_1$, this creates a cycle. 
Then we choose any edge in this cycle, except for $ij$, say $i'j'$, and remove this edge giving us a new tree $T_2$.
This is in fact equivalent with $T_n$ fulfilling the rhombus criterion. 

\begin{proposition}\label{prop: STP is rhombus}
The rhombus criterion is true for the spanning tree polytope, $T_n$.
\end{proposition}

\begin{proof}
Let $T_1$ and $T_2$ be two trees such that $v_{T_1}$, $v_{T_2}$ is a non-edge of $T_n$.
We claim that we can choose edges $ij\in T_1\setminus T_2$ and $i'j'\in T_2\setminus T_1$  such that both $T_1\cup \{i'j'\}\setminus \{ij\}$ and $T_2\cup \{ij\}\setminus \{i'j'\}$ are both trees. 

Take any edge $i'j'\in T_2\setminus T_1$. 
Then $T_1\cup\{i'j'\}$ contains a unique cycle, let us call it $C$. 
Consider the set of edges $\{(ij)_k\}$ in $C\setminus T_2$. 
If for any of these edges $T_2\cup \{(ij)_k\}\setminus \{i'j'\}$ is a tree, then we are done. 
Otherwise there must be a unique cycle $C_k$ in $T_2\cup \{(ij)_k\}$ for all $k$. 
Let $P_k\coloneqq C_k\setminus (ij)_k$. 
Then we can walk along $C$ in $T_2$, taking the route of $P_k$ whenever we must, giving us a cycle in $T_2$, a contradiction. 
Hence there is at least one edge $ij\in C\setminus T_2$ such that $T_1\cup \{i'j'\}\setminus \{ij\}$ and $T_2\cup \{ij\}\setminus \{i'j'\}$ are both trees. 

Now we have $v_{T_1\cup \{i'j'\}\setminus \{ij\}} + v_{T_2\cup \{ij\}\setminus \{i'j'\}} = v_{T_1}+e_{i'j'} -e_{ij} + v_{T_2} +e_{ij}-e_{i'j'} =  v_{T_1}+ v_{T_2}$, and what is left to show is that $v_{T_1\cup \{i'j'\}\setminus \{ij\}}\neq v_{T_2}$. 
However if this was the case then $v_{T_1}$, $v_{T_2}$ would an edge of $T_n$, a contradiction. 
\end{proof}

Any neighborly polytope, that is every pair of vertices is an edge, also fulfills the rhombus criterion.
That is however a rather uninteresting case. 
Notice that a polytope fulfilling the rhombus criterion is not stable under combinatorial equivalence of $P$, it is however direct that it is stable under affine equivalence, as affine maps preserve sums. 
\begin{proposition}
\label{prop: affine equivalence of rhombus criterion}
Let $P$ and $Q$ be two affinely equivalent polytopes. 
Then $P$ fulfills the rhombus criteria if and only if $Q$ does.     
\end{proposition}

Let us take a deeper look into how the rhombus criterion acts. 
Central here is the following theorem. 

%---THEOREM: Edge of Face
\begin{theorem}\cite{Zie95}
\label{thm: edge of face}
Let $P$ be a polytope with vertices $V$, let $\alpha,\ \beta\in V$, and assume $F$ is a face of $P$ containing $\alpha$ and $\beta$. 
Then the pair $\alpha, \beta$ is an edge of $F$ if and only if it is an edge of $P$. 
\end{theorem}

Using \cref{thm: edge of face} we also have that faces of polytopes fulfilling the rhombus criteria fulfill the rhombus criteria. 
\begin{proposition}
\label{prop: face of rhombus criterion}
Let $P$ be a polytope fulfilling the rhombus criterion and let $F$ be a face of $P$. 
For any vertices $\alpha$, $\beta$ in $F$ that have witnesses $\alpha'$ and $\beta'$ in $P$, we have $\alpha', \beta'\in F$. 
Thus $F$ fulfills the rhombus criterion. 
\end{proposition}

\begin{proof}
Let $\alpha$, $\beta$ be a non-edge of $F$. 
By \cref{thm: edge of face} $\alpha$, $\beta$ is a non-edge of $P$ and hence there are two witnesses $\alpha'$ and $\beta'$ of the rhombus criterion. 
We claim that $\alpha'$ and $\beta'$ must be in $F$. 
Indeed, as $F$ is a face there is a cost function $c$ maximizing in $F$. 
This gives us $\langle c, \alpha'\rangle \leq \langle c, \alpha\rangle= \langle c, \beta\rangle\geq  \langle c, \beta'\rangle$. 
However, by the rhombus criterion we have that $\langle c, \alpha\rangle + \langle c, \beta\rangle=\langle c, \alpha'\rangle+\langle c, \beta'\rangle$ and thus if $
\langle c, \alpha'\rangle < \langle c, \alpha\rangle= \langle c, \beta\rangle$ we must have $\langle c, \alpha\rangle= \langle c, \beta\rangle < \langle c, \beta'\rangle$, a contradiction. 
Thus we must have $\langle c, \alpha'\rangle = \langle c, \alpha\rangle= \langle c, \beta\rangle= \langle c, \beta'\rangle$ and the result follows.
\end{proof}

It can be checked that the only 2-dimensional polytopes that fulfill the rhombus criterion are triangles and rhombuses. 
Therefore we get the following corollary. 

\begin{corollary}
\label{cor: 2 dim faces}
If $P$ fulfill the rhombus criterion then the 2-dimensional faces are all triangles or rhombuses. 
\end{corollary}

% The above examples are all rather symmetrical polytopes. 
% There are however well-known examples of highly symmetric polytopes that do not fulfill the rhombus criterion. 
% Let us consider a well-know example that does not fulfill the rhombus criterion. 
In \cref{sec: examples} most examples will have some type of symmetry, therefore let us take a highly symmetric polytope that do not fulfill the rhombus criterion. 

\begin{example}
Let $P_n$ be the permutaheron of order $n$. 
That is let $v=(1, 2, \dots, n)$ and let all permutations on $\mathfrak{S}_n$ act on this vector via permuting the indices. 
Then $P_n$ is the convex hull of $\pi v$ for all $\pi\in\mathfrak{S}_n$. 
It is known that $P_n$ is a $n-1$ dimensional polytope and two vertices $\alpha$, $\beta$ is an edge if and only if there is a fundamental transposition $\pi$, that is $\pi = (i\ i+1)$ for some $i$, such that $\pi\alpha = \beta$, or vice versa. 

Then $P_3$ is a two dimensional polytope with $6$ vertices, that is a hexagon. 
By \cref{cor: 2 dim faces} we have that $P_3$ does not fulfill the rhombus criterion.  
\end{example}

Less obvious is that fulfilling the rhombus criterion is stable under cross products. 
\begin{proposition}
\label{prop: products of rhombus criterion}
Let $P$ and $Q$ be two polytopes. 
Then $P\times  Q$ fulfills the rhombus criterion  if and only if $P$ and $Q$ fulfill the rhombus criterion. 
\end{proposition}

\begin{proof}
Assume $P$ and $Q$ fulfill the rhombus criterion. 
Any edge of $P\times  Q$ is on the form $e_P\times v_Q$ or $v_P\times e_Q$ for some edge and vertex in $P$ and $Q$. 
Thus take two vertices of $P\times  Q$, $\alpha_P\times \alpha_Q$ and $\beta_P\times \beta_Q$ that is not an edge. 
If $\alpha_P=\beta_P$ (or $\alpha_Q=\beta_Q$) we must have that $\alpha_Q$, $\beta_Q$ is not an edge of $Q$, else  $\alpha_P\times \alpha_Q$, $\beta_P\times \beta_Q$ is an edge of $P\times Q$. 
As $Q$ fulfilled the rhombus criterion we can find $\alpha'_Q$, $\beta'_Q$ such that $\alpha_Q+\beta_Q=\alpha'_Q+\beta'_Q$, and therefore we have $\alpha_P\times\alpha_Q+\beta_P\times\beta_Q=\alpha_P\times\alpha'_Q+\beta_P\times\beta'_Q$. 
That is $\alpha_P\times\alpha'_Q$, and $\beta_P\times\beta'_Q$ are witnesses of the rhombus criterion. 

Otherwise we have $\alpha'=\alpha_P\times \beta_Q$ and $\beta'=\alpha_Q\times \beta_P$ as witnesses since 
$\alpha_P\times \alpha_Q+\beta_P\times \beta_Q =\alpha_P\times \beta_Q+\alpha_Q\times \beta_P$. 
Thus $P\times Q$ fulfill the rhombus criterion.

The other way around, note that $P$ is affinely equivalent to $v_Q\times P$ for any vertex $v_Q$ of $Q$. 
It follows from the above that $v_Q\times P$ is a face of $P\times Q$ and hence fulfill the rhombus criterion. 
Then the result follows by \cref{prop: face of rhombus criterion} and \cref{prop: affine equivalence of rhombus criterion}.
\end{proof}

Then as simplices are neighborly it follows that products of simplicies fulfill the rhombus criterion. 
The most well known example of this is of course the $d$-cube. 

\begin{corollary}\label{cor: Cube is rhombus}
The standard cube, $[0,1]^d$, fulfill the rhombus criterion. 
\end{corollary}

%%%%%%%%%%%%%%%%%%%%%%%%%%%%%%%%%%%%%%%%%%%%%
% \section{Chordal graph polytope}
% \label{sec: cim polytopes}
% INTRO

% The reason for we have for stydying these polytopes from the perspective of the rhombus criterion is the following theorem that was recently shown. 

% \begin{theorem}\cite{LRS22}
% Let $G$ be a tree. 
% Then $\CIM_G$ fulfills the rhombus criterion. 
% \end{theorem}

% Furthermore, in \cite{XY12} Xi and Yoshida studied the \emph{diagnosis model} and the more general \emph{fixed order $\CIM$-polytope}. 
% Fix an order $\sigma=(v_1,\dots,v_n)$ of $[n]$.
% We say that a DAG $\GG$, with skeleton $G$, respects $\sigma$ if $v_iv_j\in G$, with $i<j$, implies $v_i\to v_j\in \GG$. 
% Then given two graphs $D\subseteq H$ we define the fixed order $\CIM$-polytope to be convex hull of all DAGs $\GG$ such that $D\subseteq G\subseteq H$ and $\GG$ respects $\sigma$. 
% Then again, the rhombus criterion is fulfilled. 

% \begin{proposition}
% The fixed order $\CIM$-polytope fulfills the rhombus criterion. 
% \end{proposition}

% \begin{proof}
% By the work of \cite{XY12}, the fixed order $\CIM$-polytope are products of simplices. 
% Hence the result follows from \cref{prop: products of rhombus criterion} and the fact that simplicies are neighborly. 
% \end{proof}

% This immediately raises the question which other classes of DAGs fulfill the rhombus criterion. 
% Put another way, let $\mathfrak{D}$ be a subset of all DAGs on $n$ vertices. 
% What can we say about the polytope 
% \[
% \conv(c_\GG\colon \GG\in\mathfrak{D})?
% \]

%%%%%%%%%%%%%%%%%%%%%%%%%%%%%%%%%%%%%%%%%%%%%
\section{Characteristic Imset Polytopes}
\label{sec: cim polytopes}
The characteristic imset polytope was introduced by Studen\'y, Hemmecke, and Lindner as a 0/1-polytope and transformed causal discovery into a linear program \cite{SHL10}. 
Since its introduction it has been studied by several groups of people \cite{LRS20, S15, SV08, CS17, XY12}. 
Given any \emph{directed acyclic graph} (DAG), $\GG$, we define the \emph{characteristic imset}, $c_\GG$, as
\[
c_\GG(S) \coloneqq
\begin{cases}
1	&	\text{ if there exists $i\in S$ such that, $S\subseteq \pa_\GG(i)\cup \{i\}$},	\\
0	&	\text{ otherwise}.	\\
\end{cases}
\]
Here $\pa_\GG(i)$ denotes the \emph{parents} of $i$, that is all vertices $j$ such that we have the arc $j\to i$ in $\GG$. 
Formally, $c_\GG$ is a function $c_\GG\colon \{S\subseteq[n]\colon |S|\geq 2\}\to \mathbb{R}$, however, as $\{S\subseteq[n]\colon |S|\geq 2\}$ is finite we can identify it with a vector in $\mathbb{R}^{2^n-n-1}$. 
Then we consider the \emph{characteristic imset polytope} defined as the convex hull of all characteristic imsets,
\[
\CIM_n\coloneqq\conv(c_\GG\colon \GG=([n], E)\text{ is a DAG}). 
\]
The main question we ask is whether the rhombus criterion is true for $\CIM_n$. 
\begin{question}
Does the rhombus criterion hold true for the characteristic imset polytope $\CIM_n$?
\end{question} 
In \cref{sec: computations} we describe how the above question can be answered in the affirmative for $n\leq 5$. 
Several articles have focused on the subpolytope spanned by the imsets of a strict subset of all DAGs. 
For example, the authors of \cite{LRS20} defined 
\[
\CIM_G\coloneqq\conv(c_\GG\colon \GG\text{ is a DAG with skeleton }G). 
\]
It can in fact be shown that for any graph $G$, the polytope $\CIM_G$ is a face of $\CIM_n$ \cite[Corollary 2.5]{LRS20}. 
The reason we study these polytopes from the perspective of the rhombus criterion is the following theorem that was recently shown. 

\begin{theorem}\label{thm: CIMG is rhombus}\cite{LRS22}
Let $G$ be a tree. 
Then $\CIM_G$ fulfills the rhombus criterion. 
\end{theorem}

Furthermore, in \cite{XY12} Xi and Yoshida studied the \emph{diagnosis model} and the more general \emph{fixed order $\CIM$-polytope}. 
Fix an order $\sigma=(v_1,\dots,v_n)$ of $[n]$.
We say that a DAG $\GG$, with skeleton $G$, respects $\sigma$ if $v_iv_j\in G$, with $i<j$, implies $v_i\to v_j\in \GG$. 
Then given two graphs $D\subseteq H$ we define the fixed order $\CIM$-polytope to be convex hull of all DAGs $\GG$ such that $D\subseteq G\subseteq H$ and $\GG$ respects $\sigma$. 
Then again, the rhombus criterion is fulfilled. 

\begin{proposition}\label{prop: fixed order CIM is rhombus}
The fixed order $\CIM$-polytope fulfills the rhombus criterion. 
\end{proposition}

\begin{proof}
By the work of \cite{XY12}, the fixed order $\CIM$-polytopes are products of simplices. 
Hence the result follows from \cref{prop: products of rhombus criterion} and the fact that simplicies are neighborly. 
\end{proof}

%---SUBSECTION: CHORDAL GRAPH POLYTOPE
% \subsection{The undirected graph polytope}
% \label{subsec: chordal graph polytope}
Similar to DAGs, using undirected graphs to study causality and correlation can be done as well. 
% encoding CI statements there is also an encoding for the undirected graphs. 
The interpretation of directed and undirected graphs coincide exactly when a DAG $\GG$ is without v-structures. 
For a DAG to have no v-structures, its skeleton must be chordal, and if the skeleton is chordal we can direct it without creating v-structures. 
Thus chordal graphs exactly captures the intersection between directed and undirected graphical models. 
In this case we can consider the the characteristic imset as a vector that only depends on the skeleton of the DAG, and the definition of characteristic imset simplifies significantly. 
Given an undirected chordal graph $G$, the characteristic imset is defined as
\[
c_G(S)=\begin{cases}
1 & \text{ if  $G|_S$ is complete}\\
0 & \text{ otherwise.}
\end{cases}
\]
% Equivalently, $c_G$ is the 
% Then we define the \emph{undirected graph polytope} as 
% \[
% \UGP_n = \conv\left(c_G\colon G \text{ is a graph}\right).
% \]
% While the notation is confusingly similar to that of the characteristic imset of a directed graph, there is a good reason for this. 
% Let $\GG$ be a DAG without any v-structures and skeleton $G$, then $c_\GG=c_G$.
% Of particular interest is when these two notations coincide. 
Then we define the \emph{chordal graph polytope} as 
\[
\CGP_n = \conv\left(c_G\colon G \text{ is a chordal graph}\right).
\]
 
Thus the chordal graph polytope is spanned by the subset of all vertices of the characteristic imset polytope that does not contain v-structures. 
This polytope is well studied \cite{CS17, SCK21} from the perspective of causal discovery and describing certain facets. 
Studen\'y, Cussens, and Kratochv\'il raised the question of describing the edges of $\CGP_n$, and to this end we conjecture that the answer is in the rhombus criterion. 
\begin{conjecture} \label{conj:chordal}
The chordal graph polytope, $\CGP_n$, fulfills the rhombus criterion. 
\end{conjecture}

The rest of this section is devoted to provide evidence that \cref{conj:chordal} is true, with the strongest evidence showing that almost all pairs of chordal graphs fulfill the rhombus criterion (see \cref{thm:split edges}). 
In \cref{sec: computations} we provide a way to check that the above conjecture is true for $n\leq 6$. 
A first indication that the rhombus criterion might hold for this polytope is that it is a generalization of the spanning tree polytope that does fulfill the rhombus criterion as shown in \cref{prop: STP is rhombus}. 
\begin{proposition}
The spanning tree polytope $T_n$ is a face of $\CGP_n$. 
\end{proposition}

\begin{proof}
It is enough to find a cost function maximizing exactly in $T_n$. 
Consider the face given by maximizing the following cost function: 
\[
c(S)=\begin{cases}
1 & \text{ if } |S| = 2\\
-n^2 & \text{ if } |S| \geq 3.
\end{cases}
\]
As we maximize over all chordal graphs, if we have a cycle in $G$, then we must have a triangle in $G$ and hence the negative weight $-n^2$. 
As the positive weights of $c$ sum to $\binom{n}{2}$, we then get $\langle c, c_G\rangle < 0$. 
Thus any graph maximizing $c$ is a forest. 
Then it is direct that for any forest $G=([n], E)$ we have $\langle c, c_G\rangle = |E|$. 
This is maximized exactly when $|E|=n-1$, that is $G$ is a tree. 
\end{proof}

It is direct from definition that if $S\subseteq T$ for some $|S|\geq 2$, then if $c_G(T)=1$ we have $c_G(S)=1$, as induced subgraphs of complete graphs are complete graphs.

Let $G$ and $H$ be two graphs. 
Define the graph $\Delta(G, H)$ as having vertices $S\subseteq [n]$, $|S|\geq 2$, 
such that $c_G(S)\neq c_H(S)$. % and $S$ is inclusion maximal. 
We have an edge  $ST\in \Delta(G,H)$ if and only if $S\cap T\in \Delta(G, H)$.
We will show that $\Delta(G, H)$ is a useful tool from which we can show many cases of the rhombus criterion. 
Note however, that most theorems hold for non-chordal graphs as well.

\begin{lemma}
Let $S$ and $T$ be two neighboring nodes of $\Delta(G, H)$. 
Then $c_G(S)=c_G(T)$ and $c_H(S)=c_H(T)$. 
\end{lemma}

\begin{proof}
Since $S\in\Delta$ we can by symmetry assume that $c_G(S)=1$ and $c_H(S)=0$. 
If we do not have $c_G(T)=1$ we have $c_G(T)=0$, and thus $c_H(T)=1$. 
However, as $c_G(S)=1$, this gives us $c_G(S\cap T)=1$. 
Similarly $c_H(T)=1$ gives us $c_H(S\cap T)=1$, thus  $S\cap T\notin\Delta(G,H)$.
\end{proof}
It follows that $c_G$ and $c_H$ are constant on components of $\Delta(G, H)$. 

\begin{corollary}
\label{cor: delta constant on components}
For every component $T$ of $\Delta(G, H)$ we have either $c_G(S)=1$ for every $S\in T$, or $c_H(S)=1$ for every $S\in T$. 
\end{corollary}

We let $\mathcal{M}$ be the set of components in $\Delta(G, H)$. 
Note that for every $M\in\mathcal{M}$ we have $c_G(S)=c_G(T)$ for $S, T\in M$. 
% Then for every $M\in \mathcal{M}$ we define 
% \[
% s_M(S) = \begin{cases}
% 1&\text{ if } c_G(S)=0 \text{ and } c_H(S)=1\\
% -1&\text{ if } c_G(S)=1 \text{ and } c_H(S)=0\\
% 0 &\text{ otherwise.}
% \end{cases}
% \]
% It is direct that $c_H-c_G=\sum_{M\in\mathcal{M}}s_M$. 
% Because of this $s_M$ depends on the order of $G$ and $H$, however we also see that exchanging $G$ and $H$ flips the sign of $s_M$. 
% By \cref{cor: delta constant on components} we get that on every component of $\mathcal{M}$ we have that either $s_M$ or $-s_M$ is a $0/1$-vector. 
Let $\mathcal{S}\subseteq \mathcal{M}$ and consider the imset 
\[
c_{\mathcal{S}}(S) = \begin{cases}
1&\text{ if } c_G(S)= c_H(S)=1,\\
0&\text{ if } c_G(S)= c_H(S)=0,\\
1&\text{ if } S\in T \text{ for some } T\in\mathcal{S}\text{, and }\\
0 &\text{ otherwise.}
\end{cases}
\]
Note that $c_\mathcal{S}$ agrees with $c_G$ and $c_H$ outside of $\Delta(G, H)$. 
% Let $\mathcal{S}\subseteq \mathcal{M}$ and consider the imset $c_{G, \mathcal{S}}\coloneqq c_G+\sum_{M\in \mathcal{S}}s_M$.
Small examples show that $c_{\mathcal{S}}$ need not be the characteristic imset of any graph, but it can happen. 
If there is a graph $D$ such that $c_D=c_{\mathcal{S}}$ we say that $\mathcal{S}$ is \emph{compatible}, moreover if $D$ is chordal we will call  $\mathcal{S}$ \emph{chordally compatible}. 
% The equivalent for $H$ is that there exists a (chordal) graph $D$ such that $c_D=c_{H, \mathcal{S}}\coloneqq c_H-\sum_{M\in \mathcal{S}}s_M$, note the change in sign due to the previous mentioned implicit order of $G$ and $H$. 
We see that there always exists at least two compatible sets, namely $\mathcal{S}=\{T\in\mathcal{M}\colon c_G|_T=1\}$, as we get $c_{\mathcal{S}}=c_G$, and $\mathcal{S}=\{T\in\mathcal{M}\colon c_H|_T=1\}$, as $c_{\mathcal{S}}=c_G$. 
% Since we have $c_{\mathcal{S}}=c_{H,\mathcal{M}\setminus\mathcal{S}}$ we get that $G$ is (chordally) compatible with $\mathcal{S}$ if and only if $H$ is (chordally) compatible with $\mathcal{M}\setminus\mathcal{S}$. 
Let us consider all vertices that come from (chordally) compatible sets, 
\[
C_{G, H}\coloneqq \conv(c_{G,\mathcal{S}}\colon \mathcal{S}\subseteq\mathcal{M} \text{ is chordally compatible with }G). 
\]
This is in fact a face of $\CGP_n$. 
\begin{lemma}
\label{lem: compatible face}
Let $G$ and $H$ be two chordal graphs and let $\mathcal{M}$ be the set of connected components in $\Delta(G, H)$. 
Then $C_{G, H}$ is a face of $\operatorname{CGP}_n$.
\end{lemma}

\begin{proof}
First we consider the cost function 
\[
W(S)=\begin{cases}
1 & \text{ if } c_G(S)=c_H(S)=1,\\
-1 & \text{ if } c_G(S)=c_H(S)=0,\\
0&\text{ otherwise.}
\end{cases}
\]
Then maximizing $W$ are exactly the characteristic imsets $c_D$ such that $G\cap H\subseteq D\subseteq G\cup H$. 
Notice that this includes all graphs defined by $c_{G,\mathcal{S}}$ where $\mathcal{S}$ is compatible with $G$. 
As faces of faces are faces we can henceforth restrict out focus to only this face and these graphs. 

Let $M$ be an inclusion maximal element in $\Delta(G, H)$, $\downarrow M$ denote the subsets of $M$ (including $M$) in $\Delta(G, H)$, and define $m_M=|\downarrow M|$.
If $m_M> 1$ we let 
\[
w_M(S)= \begin{cases}
1 & \text{ if } S=M,\\
\frac{-1}{m_M-1} & \text{ if } S\in \downarrow M\setminus\{M\},\\
0&\text{ otherwise.}
\end{cases}
\]
Otherwise we just set $w_M(S)=0$ for all $S$. 
Note that for any graph $D$ we have $\langle w_M, c_D\rangle\leq 0$ with equality if and only of $c_D(M)=1$ or $c_D(S)=0$ for all $S\in \downarrow M$. 
That is, for all $c_D|_{\downarrow M}$ is either $c_G|_{\downarrow M}$ or $c_H|_{\downarrow M}$. 

Take an element $T\in \mathcal{M}$ and define $w_T=\sum w_M$ where we sum over all inclusion maximal elements $M\in T$. 
Let $M, M'\in T$ be two inclusion maximal neighbors. 
If $D$ is a graph such that $c_G(M)=0$ and $c_G(M')=1$ then we get that $\langle w_M, c_D\rangle \leq w_M(M\cap M') c_D(M\cap M')=\frac{-1}{m_M-1}$ which is strictly less than $0$ if $m_M> 1$. 
However, this must be the case as $M$ was inclusion maximal with an edge to $M'$. 
Thus we have that $\langle w_T, c_D \rangle = 0$ if and only if $c_D|_T\equiv 0$ or $c_D|_T\equiv 1$. 
Then by \cref{cor: delta constant on components} gives us that $c_D|_T$ is either $c_G|_T$ or $c_G|_T$. 
Note that this is the same as $D$ being given by $c_{\mathcal{S}}$ for some $\mathcal{S}\subseteq \mathcal{M}$ such that $T\in \mathcal{S}$. 
\end{proof}

% Since all vertices $c_D\in C_{G, H}$ maximize $W$ we have $G\cap H\subseteq D\subseteq G\cup H$. 
Then we can rephrase that $\CGP_n$ fulfills the rhombus criterion as follows. 

\begin{lemma}
\label{lem: compatible witnesses}
Let $G$ and $H$ be two chordal graphs and let $\mathcal{M}$ be the set of connected components of $\Delta(G, H)$. 
Then $c_G$, $c_H$ have witnesses of the rhombus criterion if and only if there does exist a $\mathcal{S}\subseteq \mathcal{M}$ such that $\mathcal{S}$ and $\mathcal{M}\setminus\mathcal{S}$ are chordally compatible and $c_\mathcal{S}$ is not $c_G$ or $c_H$. 
\end{lemma}

\begin{proof}
The proof of \cref{prop: face of rhombus criterion} it is enough to search after witnesses in any face containing $c_G$ and $c_H$. 
Thus by \cref{lem: compatible face} is enough to consider all vertices $c_{\mathcal{S}}$ such that $\mathcal{S}\subseteq\mathcal{M}$ is compatible with $G$. 
Moreover, for any $\mathcal{S}$ we have $c_{\mathcal{S}} + c_{\mathcal{M}\setminus\mathcal{S}} = c_G+c_H$. 
Thus we have witnesses to the rhombus criterion if and only if there is some $\mathcal{S}$ and $\mathcal{M}\setminus\mathcal{S}$ is chordally compatible and $c_\mathcal{S}$ is not $c_G$ or $c_H$. 
\end{proof}

\begin{corollary}
\label{cor: 4 vertices}
If there are 4 of fewer chordally compatible sets for $G$ and $H$, then either $c_G$, $c_H$ fulfill the rhombus criterion. 
\end{corollary}

\begin{proof}
Any 4 (or fewer) 0/1-vectors are either a rhombus or a simplex, both fulfill the rhombus criterion. 
By assumption the face of \cref{lem: compatible face} must then fulfill the rhombus criterion and by \cref{prop: face of rhombus criterion} the result follows. 
\end{proof}

Note that whenever $\Delta(G, H)$ has 2 or fewer components we can immediately use the above corollary. 
Thus the question that we must answer is how does the (chordally) compatible subsets look like. 
Define the graph \emph{induced} by $c_\mathcal{S}$ as the graph $D$ where we have an edge $i,j\in D$ if and only if $c_{\mathcal{S}}(\{i,j\})=1$. 
For convenience, for any component $S$ in $\Delta(G, H)$ we let $E(S)\coloneqq\{T\subseteq S\colon |T|=2\}$, that is, $E(S)$ is the set of edges in $S$. 
Therefore, the graph induced by $c_\mathcal{S}$ can be described as $(G\cap H)\cup\bigcup_{S\in\mathcal{S}}E(S)$. 
We observe the following lemma. 

\begin{lemma}
\label{lem: compatibilty criterions}
Let $G$ and $H$ be two graphs, $\mathcal{S}\subseteq\mathcal{M}$, and define $D$ as the graph induced by $c_\mathcal{S}$.  
If $\mathcal{S}$ is compatible, then $c_\mathcal{S}=c_D$. 
If $\mathcal{S}$ is not compatible, then $D$ has a set $S$ such that $c_D(S)=1$ and $c_{\mathcal{S}}(S)=0$. 
Moreover, if $G$ and $H$ are chordal graphs and $\mathcal{S}$ is not chordally compatible, then either 
\begin{enumerate}
    \item there is a cycle without chord in $D$, or
    \item there is a set $S$ such that $c_D(S)=1$ and $c_{\mathcal{S}}(S)=0$. 
\end{enumerate}
\end{lemma}

Note that condition (2) is the only condition for (non chordal) compatibility, thus the chordal case can be seen as a subcase of the general case where we also have to make sure that $D$ is chordal.  

\begin{proof}
We will only show the chordal case as that implicitly includes the non-chordal case. 
It follows from definition that any imset $c$ is the imset of a graph $D'$ the following must hold:
\begin{enumerate}[label=\roman*)]
    \item we have an edge $ij\in D'$ if and only of $c(\{i,j\})=1$, 
    \item for every set $S$ of size 3 or greater, $c(S)=1$ if and only if $S$ is complete in $D'$.
\end{enumerate}
Therefore it follows that $\mathcal{S}$ is compatible, then $c_\mathcal{S}=c_D$. 
By the first condition, if $\mathcal{S}$ is compatible with $G$ then we must have $c_{\mathcal{S}}=c_D$. 
By definition of $c_{\mathcal{S}}$ we have that if $c_{\mathcal{S}}(S)=1$, then $c_{\mathcal{S}}(S')=1$ for all subsets $S'\subseteq S$. 
Thus all that can fail is either the second condition or that $D$ is not chordal. 
If $D$ is not chordal we are done as (1) follows.
Thus if compatibility is to fail, it must be because we have a set $S$ that is complete in $D$, but $c_{\mathcal{S}}(S)=0$, thus (2) follows.
\end{proof}

\begin{proposition}
\label{prop: subgraph rhombus criterion}
Let $G$ and $H$ be two chordal graphs such that $G\subseteq H$. 
Then the vertices $c_G$, $c_H$ of $\CGP_n$ fulfill the rhombus criterion. 
\end{proposition}

\begin{proof}
If $\Delta(G, H)$ is connected we are done by \cref{cor: 4 vertices} as the face only contains $c_G$ and $c_H$. 
Thus we can assume that $\Delta(G, H)$ has components $\{S_1,\dots, S_k\}=\mathcal{M}$ where $k\geq 2$. 
We claim that any $\mathcal{S}\subseteq \mathcal{M}$ is chordally compatible. 
Since $G\subseteq H$ we have $c_G|_{S_i}=0$ for all $i$, hence the only way $\mathcal{S}$ and $G$ are not compatible will be if the graph induced by $c_\mathcal{S}$, $D$, will be if $D$ has a chordless cycle. 
Assume we have a cycle $C\subseteq D$. % is such a cycle.
As this cycle is not in $G$, we must have added in an edge $ij\in C\subseteq H$. 
However, as $H$ is chordal we have a chord $i-j-\ell-i$ for a $\ell\in C$. 
As this chord is not present in $G$, both the chord, the triple $\{i,j,\ell\}$, and the edge $ij$ must be in $\Delta(G, H)$. 
By definition of $\Delta(G, H)$ they are in the same connected component, and thus we must have $c_D(\{i,j,\ell\})=c_D(\{i,j\})=1$. 
Thus $C$ has a chord, and it follows that $\mathcal{S}$ is chordally compatible. 
Therefore the face $C_{G, H}$ defined in \cref{lem: compatible face} is a cube, and if we let $\mathcal{S}=\{S_1\}$ we get our witnesses. 
\end{proof}
 
\begin{corollary}
The complete graph $K_n$ has an edge to every other vertex of $\CGP_n$. 
Therefore the diameter of $\CGP_n$ is 2 if $n\geq 3$. If $n=2$, the diameter is $1$, and if $n=1$ it is $0$. 
\end{corollary}

\begin{proof}
% Let $H$ be the complete graph on $[n]$. 
For any chordal graph $G\neq K_n$ we have $G\subseteq K_n$ with $[n]\in\Delta(G, H)$. 
It follows by definition that $[n]$ is connected to every node in $\Delta(G, H)$ and thus, as noted in the proof of \cref{prop: subgraph rhombus criterion}, $c_G$, $c_H$ is an edge of $\CGP_n$.
Therefore the diameter of $\CGP_n$ is at most two.

If $n\geq 3$ we need to find two graphs that does not form an edge of $\CGP_n$. 
To this end, consider two graphs each containing a single edge. 
Then the union of these graphs and the graph with no edges are witnesses of the rhombus criteria. 
All that is left are the cases $n\leq 2$, all of which can be done by hand. 
\end{proof}

The first part in the above corollary also follows from the observation that the complete graph is the only graph such that $c_G([n])=1$.
Hence the chordal graph polytope is a pyramid with $c_{K_n}$ as apex. 

A graph is called a \emph{split graph} if $V_G=K_G\dot\cup O_G$ where $O_G$ is an independent set in $G$ and $K_G$ is complete.
Equivalently, split graphs are exactly the chordal graphs whose complement is also chordal. The split into the disjoint set $K_G$ and $O_G$ is not unique in general, and we will assume that $K_G$ is maximal. 
Split graphs is a very interesting class of chordal graphs, and for sufficiently large $n$ they are almost all chordal graphs. 

\begin{theorem}\cite{BRW85}
Let $s_n$ be the number of split graphs with $n$ vertices and let $c_n$ be the number of chordal graphs. 
Then for any $\alpha > \frac12\sqrt{3}$ and $n$ sufficiently large we have
\[
\frac{s_n}{c_n}> 1-\alpha^n.
\]
\end{theorem}
That is, asymptotically almost all chordal graphs are split graphs.
We will often use that a split graph $G$ is completely determined by $O_G$, and $\ne_G(x)$ for all $x\in O_G$. 

\begin{lemma}
\label{lem: complete neigh}
Let $G$ and $H$ be two chordal graphs. 
If there is a node $x$ such that $\ne_G(x)\neq\ne_H(x)$ and both are complete in both $G$ and $H$ then $c_G$, $c_H$  fulfills the rhombus criterion. 
Moreover, if $G$ and $H$ only differ on edges incident to $x$, then $c_G$, $c_H$ is an edge, otherwise we have witnesses. 
\end{lemma}

\begin{proof}
We will continuously use that if we let $D$ be a graph with a node $y$ such that $\ne_D(y)$ is complete in $D$. 
Then $D$ is chordal if and only if $D|_{[n]\setminus \{x\}}$ is chordal. 
This can be shown, for example, using a standard argument about perfect elimination orderings.

Define the graph $G'$ from $G$ via removing all edges connected to $x$ and then adding in edges such that $\ne_{G'}(x)=\ne_H(x)$. 
Similarly define $H'$ such that  $\ne_{H'}(x)=\ne_G(x)$. 
By the above, both $G'$ and $H'$ are chordal. 

If we let $\Gamma_G(x)=\{S\cup \{x\}\colon S\subseteq\ne_G(x), |S|\geq 1\}$, and define $\Gamma_G(x)$ similarly. 
It is direct that 
\[
c_{G'}=c_G-\sum_{S\in \Gamma_G(x)}e_S+\sum_{S\in \Gamma_H(x)}e_S, 
\]
and again, similarly for $c_{H'}$. 
It follows that $c_G+c_H=c_{G'}+c_{H'}$. 
If $G'\neq H$, then the rhombus criterion is fulfilled, with $c_{G'}$ and $c_{H'}$ as witnesses. 

Otherwise, it follows that $\Delta(G, H)=(\Gamma_G(x)\setminus \Gamma_H(x))\cup (\Gamma_H(x)\setminus \Gamma_G(x))$. 
If $\Gamma_G(x)\setminus \Gamma_H(x)$ is non-empty, we must have $\{x\}\cup\ne_G(x)\in \Gamma_G(x)\setminus \Gamma_H(x)$. 
As every set in $\Gamma_G(x)\setminus \Gamma_H(x)$ is a subset of $\{x\}\cup\ne_G(x)$ it follows by definition of $\Delta(G, H)$ that $\Gamma_G(x)\setminus \Gamma_H(x)$ is a (possibly empty) connected component of $\Delta(G, H)$. 
By symmetry the same is true for $\Gamma_H(x)\setminus \Gamma_G(x)$. 
Therefore $C_{G, H}$ has at most $4$ vertices, and the result follows from \cref{cor: 4 vertices}. 

Going even further, if $c_G$ and $c_H$ is not an edge of $C_{G, H}$, then we must have $4$ compatible sets in $\Delta(G, H)$. 
Since we only had 2 components in $\Delta(G, H)$, the compatible sets must be $\{\emptyset, \Gamma_H(x)\setminus \Gamma_G(x), \Gamma_G(x)\setminus \Gamma_H(x), \Delta(G, H)\}$. 
Two of these sets represent $G$ and $H$, and the other two $G'$ and $H'$.  %  $\emptyset$ and $\Delta(G, H)$ gives us $c_G$ and $c_H$, respectively.
% By the computations above, the other two graphs are $G'$ and $H'$.
Therefore, if $c_{G'}$ and $c_{H'}$ are not witnesses, then $c_G$, $c_H$ is an edge. 
\end{proof}

% For split graphs we want a local version of the $\Delta(G, H)$ graph to dentote the local change around a node $x$. 
% Therefore we let $E_G(x)\coloneqq \{S\cup \{x\}\colon S\subseteq \ne_G(x) \textrm{ and } |S|\geq 1\}$ and define $\delta_G(X)\coloneqq E_G(x)\cap \Delta(G, H)$, and similar with $\delta_H(x)$. 

\begin{theorem}\label{thm:split edges}
Assume $G$ and $H$ are two split graphs. 
Then $c_G$ and $c_H$ fulfill the rhombus criterion in $\CGP_n$. 
\end{theorem}

\begin{proof}
Up to symmetry we have 3 cases $K_G=K_H$, $K_G\subsetneq K_H$, and $K_G\centernot\subseteq K_H$.
We will let $\Gamma_G(x)\coloneqq \{S\cup \{x\}\colon S\subseteq \ne_G(x) \textrm{ and } |S|\geq 1\}$. 
Note that if $\ne_G(x)$ is complete, then $\delta_G(x)\coloneqq \Gamma_G(x)\cap \Delta(G, H)$ is connected (albeit, maybe empty) in $\Delta(G, H)$.  
It is direct that $\{S\in\Delta(G, H)\colon x\in S\}=\delta_G(x)\cup \delta_H(x)$. 
We also recall the definition of $E(S)$ as all edges in a component of $\Delta(G, H)$. 
For example, $E(\delta_G(x))$ consists of all edges indicent to $x$ in $G\setminus H$.
By \cref{lem: compatible face} it is enough to consider graphs $D$ such that $G\cap H\subseteq D\subseteq G\cup H$. 

\textbf{CASE I, $K_G=K_H$:} 
If $\ne_G(x)=\ne_H(x)$ for all $x\in O_G=O_H$, then $G=H$ as $G$ and $H$ are split graphs. 
Otherwise we have a node such that $\ne_G(x)\neq \ne_H(x)$. 
As $G$ and $H$ are split graphs we have that $\ne_G(x)$ and $\ne_H(x)$ both lie in $K_H=K_G$ and thus are complete in both $G$ and $H$. 
Thus this case follows by \cref{lem: complete neigh}. 

Working out the details, if there is a unique node $x\in O_G=O_H$ such that $\ne_G(x)\neq \ne_H(x)$, then $c_G$, $c_H$ is an edge, otherwise we have witnesses.

%$\ne_G(x)\neq\ne_H(x)$, then we claim the the face of \cref{lem: compatible face}, has 3 or fewer vertices and is therefore a triangle. 
% it is straightforward that $\Delta(G, H)$ has two connected components, $E_G(x)$, and $E_H(x)$. 
% Thus the face of \cref{lem: compatible face} has at most $4$ vertices and by \cref{cor: 4 vertices} we are done. 
% Going further it is possible to show that $\{E_G(x), E_H(x)\}$ is not chordally compatible with $G$, and hence $c_G$, $c_H$ is in fact an edge. 

% If there are two or more vertices such that $\ne_G(x)\neq\ne_H(x)$, fix one of them. 
% Define the graphs $G'$ as $\ne_{G'}(x)=\ne_H(x)$, and $\ne_{G'}(y)=\ne_G(y)$ for all other $y\in O_G$, and similarly define $H'$ from $H$. 
% Note that $G$, $H$, $G'$, and $H'$ are all different by assumption, and it is straightforward to see that they are all spit graphs, and hence chordal.  
% All that is left is to conclude $c_{G'}=c_G+\chi_{E_H(x)}-\chi_{E_G(x)}$
%  and $c_{H'}=c_H+\chi_{E_G(x)}-\chi_{E_H(x)}$ and it follows that $c_{G'}$ and $c_{H'}$ are witnesses of the rhombus criterion. 

\textbf{CASE II, $K_G\subsetneq K_H$:} 
Note that we have $O_H\subseteq O_G$.
\\
{\textbf a)} If $\ne_G(x)=\ne_H(x)$ for all $x\in O_H$, then we have $G\subseteq H$ and thus the result follows by \cref{prop: subgraph rhombus criterion}. 
We can furthermore note that $K_H\in \Delta(G, H)$ and thus $\Delta(G, H)$ is connected. 
It follows by \cref{lem: compatible face} that $c_G$, $c_H$ is an edge. 
% {\color{red} Here I redefined the cases as it is more direct. Note that if $\ne_H(x)$ is complete in $G$, it follows that $\ne_H(x)\cap (K_H\setminus K_G)$ is a most one.}
\\[1mm]
{\textbf b)} If there is a node $x\in O_H$ such that $\ne_H(x)\neq \ne_G(x)$ and $\ne_H(x)$ is complete in $G$ (it is always complete in $H$) we can apply \cref{lem: complete neigh}, and we are done. 
Note that this includes all the cases when $|\ne_H(x)|=0$ or $1$.  Since $K_G\neq K_H$ we always have that $G'\neq H$ and thus we always get witnesses. 
% indeed repeat the same construction as in \textbf{Case I} and thus we have witnesses as it cannot happen that $G'=H$ because $K_{G'}=K_G\neq K_H$. 
\\[1mm]
{\textbf c)}
Otherwise, for every node $x\in O_H$ such that $\ne_H(x)\neq \ne_G(x)$ we have that $\ne_H(x)$ is not complete in $G$. We note that due to $K_G\subsetneq K_H$ we must have that $K_H$ is not complete in $G$ and hence $K_H\in\Delta(G, H)$. Therefore we can define $S_H$ to be the  component of $\Delta(G, H)$ containing $K_H$. 
Consider $\delta_H(x)$ for any $x\in O_H$ such that $\ne_H(x)\neq \ne_G(x)$. 
As $\ne_H(x)$ is not complete in $G$, there is a set $\{z,y\}\in\Delta(G, H)$, that is $zy$ is an edge in $H$ but not in $G$, where $z,y\in \ne_H(x)$. 
Thus  we have an edge between $K_H$ and $\ne_H(x)$ in $\Delta(G, H)$ and as $\delta_H(x)$ is connected in $\Delta(G, H)$, it follows that $\delta_H(x)$ is part of $S_H$. 
Every edge in $H$ not in $G$ is either in $K_H$ or in $\delta_H(x)$ for some $x\in O_H$. Therefore, every set $S\in \Delta(G, H)$ such that $c_H(S)=1$ is in $S_H$. 

Let us consider any $\mathcal{S}\subseteq \mathcal{M}$ with $S_H\in \mathcal{S}$. 
% Note that since $S_H\in\mathcal{S}$ we must have that if $c_H(S)=1$ or $c_G(S)=1$ we have that $c_{G,\mathcal{S}}(S)=1$, because $S_H$ contained all sets in $\Delta(G, H)$ where $c_H(S)=1$.
If $\mathcal{M}=\{S_H\}$, then by \cref{cor: 4 vertices} we are done, thus assume $|\mathcal{M}|\geq 2$. 
We claim that if $\mathcal{S}\neq \mathcal{M}$, then $\mathcal{S}$ breaks criterion (2) of \cref{lem: compatibilty criterions}. 
To prove this, take any $T\in\mathcal{S}\setminus\{S_H\}$. 
By the above we have that $c_G|_T=1$ and therefore $T$ contains at least one edge not present in $H$. 
As $K_G\subsetneq K_H$, and $G$ was a split graph, this edge must be on the form $a-b$ where $a\in O_H$ and $b\in K_G$. 
Therefore $\ne_G(a)\neq \ne_H(a)$ and we can consider $N=\ne_G(a)\cup\ne_H(a)\cup\{a\}$. 
As $a-b\notin H$ we must have $c_H(N)=0$. 
If $c_G(N)=1$ then $\ne_H(a)$ is complete in $G$, which was dealt with above, hence we can assume $c_G(N)=0$. 
Thus $N\notin \Delta(G, H)$ and it follows that $c_{G,\mathcal{S}}(N)=0$. 
However, since $S_H\in \mathcal{S}$ we must have that $c_{G,\mathcal{S}}(S)=1$ for every 2-set $S\subseteq K_H$. 
Moreover, since $c_H(\ne_H(a)\cup\{a\})=1$ and $c_G(\ne_G(a)\cup\{a\})=1$, we must have $c_{G,\mathcal{S}}(\{a, y\})=1$ for every $y\in \ne_G(a)\cup\ne_H(a)$. 
Thus by \cref{lem: compatibilty criterions}, $\mathcal{S}$ is not chordally compatible as it breaks criterion (2). 

Thus, of all chordal graphs $D$ such that $c_D=c_{G,\mathcal{S}}$ for some chordally compatible $\mathcal{S}$, $H$ is the only graph with $c_D|_{S_H}=1$.
It follows that $c_H$ is adjacent to every vertex in the face given in \cref{lem: compatible face} as it is a pyramid over all the other vertices. 
This finishes this case. 

\textbf{CASE III, $K_G\centernot\subseteq K_H$:} 
By symmetry we can assume that $K_H\centernot\subseteq K_G$. 
\\[1mm]
{\textbf a)} If there is a vertex $x\in O_H\cup O_G$ such that $\ne_G(x)\neq \ne_H(x)$, $\ne_G(x)$ is complete in $H$ and $\ne_H(x)$ is complete in $G$ we can apply \cref{lem: complete neigh} as $\ne_G(x)$ and $\ne_H(x)$ are complete in $G$ and $H$ respectively because $G$ and $H$ are split graphs. If there is only one such vertex $x\in O_H\cup O_G$, then $c_G,c_H$ forms polytopal edge, otherwise we get witnesses.
\\[3mm]
With $S_G$ and $S_H$ as defined in Case II, we can use the exact same argument as in Case IIc to show that
if $\ne_G(x)$ is not complete in $H$, then $\delta_G(x)\subseteq S_G$. % {\color{red} This lemma will be so technical it is close to unreadable, do we really want it?}
As $c_G$ (and $c_H$) differ in value on $\delta_G(x)$ and $\delta_H(x)$ we cannot have $\delta_G(x)$ and $\delta_H(x)$ in the same connected component by \cref{cor: delta constant on components}. 
Therefore, the connected parts of $\Delta(G, H)$ are $S_G$, $S_H$, and instances of $\delta_G(x)$ for $x\in O_G$ and $\delta_H(x)$ for some $x\in O_H$.% , joined into components in some way. {\color{red} How can they be joined?} 
Note that $\delta_G(x)$ is a separate component if and only if $\ne_G(x)$ is complete in $H$, and vice versa with $\delta_H(x)$. 
We may assume $\ne_G(x)\neq\ne_H(x)$ since otherwise $\delta_G(x)=\emptyset$.
\\[1mm]
{\textbf b)} In this case we will study when $\{S_G,S_H\}$ is a compatible set. 
To do this we define an edge $e\in G\cup H $ to be {\bf red} if either it is present in both $G$ and $H$, or if it belongs to $E(\delta_G(x))$ for some $x$ such that $\delta_G(x)\subseteq S_G$ or it belongs to $E(\delta_H(x))$ for some $x$ such that $\delta_H(x)\subseteq S_H$. 
Note that for any $\mathcal{S}$ with $\{S_G,S_H\} \subseteq \mathcal{S}$ we have that the graph induced by $c_\mathcal{S}$ includes all red edges.
Other edges that are not red are called {\bf blue}. Edges present in both $G$ and $H$ have two copies, one red and one blue.  % copies -> colors 
Note that an edge is blue if and only if it is either present in both $G$ and $H$ or it is an edge $ok$ with $o\in O_H$ (or $O_G$) with $\ne_H(o)$ complete in $G$. 
We will show that either the red and blue edges give us witnesses, or we can utilize the obstruction to create witnesses. 
% If $K_G\cap K_H\neq\emptyset$ and there is any fixed edge from $K_G\setminus K_H$ to $K_H\setminus K_G$ then $\{S_G,S_H\}$ cannot both be in the same witness since a triangle is formed that is in neither $G$ nor $H$. 
% If $K_G\cap K_H\neq\emptyset$ and there are fixed edges $e,f$ with $e\in\delta_G(x), f\in\delta_H(x)$ that form a four-cycle without a chord. 
% Then $\{S_G,S_H\}$ is not compatible. (DO WE NEED TO SAY EXACTLY WHEN THIS HAPPENS?)
%Assume instead that $\{S_G,S_H\}$ is a chordally compatible set, we want to show that we then always can form witnesses, 
In either case $c_G$ and $c_H$ are not neighbors in the polytope. 
If the remaining components of $\Delta(G,H)$ together induces a chordal graph $F$ (the blue edges) with no new cliques then we have our witnesses. 
That is, if $\mathcal{S}=\Delta(G,H)\setminus\{S_G, S_H\}$ is chordally compatible, then we are done by \cref{lem: compatible witnesses}.
If this is not the case, the graph $F$ is either not chordal, or has a clique not present in $G$ or $H$. 
In the second case we then have a triangle not present in $G$ or $H$, which is a 3-cycle, and hence, if $\Delta(G,H)\setminus\{S_G, S_H\}$ is not chordally compatible, we have a cycle in $F$ not present in $G$ or $H$. Let $C$ be a minimal such cycle. 
By definition of $F$, all edges in $C$ are blue. 
% that are not edges in $K_G$ or $K_H$

If $C$ contains a node $o\in O_G\cap O_H$, then we have some path $hog\in C$. 
Then the edges $ho$ and $og$ are both blue. 
Assume $h, g\in\ne_G(o)$, then as $ho$ and $og$ are both blue, we must have that $\delta_G(o)$ is a separate component, that is, $\ne_G(o)$ is complete in $H$. 
Therefore we can instead consider the shorter cycle $C$ where we use the edge $hg$ that is present in both $G$ and $H$, and therefore blue. 
Similarly with $h, g\in\ne_H(o)$. 
Therefore we consider when $g\in\ne_G(o)$ and $h\in\ne_H(o)$. 
As $ho$ and $og$ are both blue we must have that both $\delta_G(o)$ and $\delta_H(o)$ are components of $\Delta(G, H)$ not in $S_G$ nor $S_H$. 
This is equivalent to $\ne_G(o)$ being complete in $H$, and $\ne_H(o)$ being complete in $G$. 
Then we can utilize \cref{lem: complete neigh} to obtain witnesses. 
Therefore we can henceforth focus on the case when $C$ has no vertices in $O_G\cap O_H$. 

%-- REWRITE FROM HERE---
% We will now define an  of $\mathcal{M}$. 
% Let $\mathcal{S}\subseteq\mathcal{M}$. {\color{red} should this include $S_G/S_H$ or not?}
% If for every $\delta_G(x)\in\mathcal{S}$ such that $\delta_G(x)\centernot\subseteq S_G$ we have $\delta_H(y)\in\mathcal{S}$ for every $y\in\ne_G(x)$, and vice versa for $\delta_H(x)\in\mathcal{S}$, then we say that $\mathcal{S}$ is an {\bf alternating set}. 
% We claim that if we have an alternating set, then we have witnesses of the rhombus criterion. 
% Indeed, assume that $\mathcal{S}$ is alternating, we wish to show that both $\mathcal{S}$ and $\mathcal{M}\setminus\mathcal{S}$ are chordally compatible. 
% \\[3mm]
% Any obstruction would be a cycle $C$ formed by free edges that is not present in neither $G$ nor $H$. 
% This either makes $F$ non-chordal or introduces a simplex not in $c_G$ or $c_H$. 
% If such a cycle were to include a node $x\in O_G\cap O_H$. Then $\ne_G(x)$ is complete in $H$ and $\ne_H(x)$ is complete in $G$ so this is case IIIa 
% By assuming $C$ is minimal we can ignore the case when $\ne_G(x)=\ne_H(x)$.  

Assume $C$ contains a node $x\in K_G\cap K_H$. 
If $C$ is contained in $K_G\cap K_H$ then it has no new cycles as both $G$ and $H$ are complete on this set and hence all edges are blue. 
Therefore assume that $xg$ is an edge of $C$ where $g$ is outside $K_G\cap K_H$. 
By symmetry we can assume $g\in K_G\setminus K_H$. 
Let $h$ be the next vertex in $C$. 
If $h\in K_G\cap K_H$ we have an edge $xh$ in $F$, and therefore $C$ is not minimal. 
If $h\in K_G\setminus K_H$ then $gh$ is not in $H$, and hence $\{g_h\}$ is in $S_G$, and therefore not blue. 
We concluded above that $C$ has no vertices in $O_G\cap O_H$.
Therefore the only option is that $h\in K_H\setminus K_G$. 

Note that $|\ne_H(g)\cap(K_H\setminus K_G)|=1$, since the edges in $C$ are blue. 
Indeed, as $C$ only has blue edges $\ne_H(g)$ is complete in $G$, and as $H$ split we can have at most one node outside of $K_H$. 
Therefore $\ne_H(g)\cap(K_H\setminus K_G)=\{h\}$. 
Since $x$ and $h$ are both in $K_H$, and we have the edge 
$xh$ in $H$. 
Moreover $\ne_H(g)$ is complete in $G$ and hence we have $xh$ in $G$ as well. 
Therefore $C$ is not minimal as we could have used the edge $xh$ instead, or $C$ is a triangle and therefore does not contradict chordality of $F$. 
Moreover $c_G(\{x,g,h\}=1$, therefore this triangle exists in $G$.

In conclusion, $C$ has no node $x\in K_G\cap K_H$. 
% The edges in $C$ may be part of larger simplex in $G$ (or $H$), but for $g\in C\cap(K_G\setminus K_H)$ we must have 
% Say $\ne_H(g)\cap(K_H\setminus K_G)=\{h\}$. 
% If $|\ne_H(g)|\ge 2$ it then follows that there is also an edge from $g$ to some $x\in K_G\cap K_H$. 
% In this case, since the edges are blue, we will have a triangle (or larger simplex) from $g$ already in $H$, so it is not part of a new cycle.
Therefore the cycle $C$ must thus consist of edges going between $K_G\setminus K_H$ and $K_H\setminus K_G$, and alternating in edges from $G$ and $H$. 
% Then $F$ is not a witness. 
In this situation we form instead two witness by $G'=G\setminus(C\cap G)\cup(C\cap H)$ and $H'=H\setminus(C\cap H)\cup(C\cap G)$. 
% Note that any red and blue edges stay in both $G'$ and $H'$. 
We use here that for each vertex in $C$ lying in $K_G\setminus K_H$ has at most one neighbor in $K_H\setminus K_G$ which is thus connected with that blue edge, and we may thus swap the edges along the alternating cycle.
\\[3mm]
{\textbf c)} In this case we will study when $\{S_G,S_H\}$ is not a compatible set. 
Similar to above, we will try to color all edges such that we either find witnesses, or in this case we might obtain an edge-defining cost function. 
% We may thus assume that if there are witnesses then $S_G$ and $S_H$ are in different witnesses. 
Color an edge $e\in G$ {\bf green} if either it is present in both $G$ and $H$, or if it belongs to $\delta_G(x)$ for some $x$ such that $\delta_G(x)\subseteq S_G$. 
Then color a edge {\bf purple} if either it is present in both $G$ and $H$, or if it belongs to $\delta_H(x)$ for some $x$ such that $\delta_H(x)\subseteq S_H$. 
The colors will denote whether the edges need to appear with $S_G$ or $S_H$ in a chordally compatible subset $\mathcal{S}$. 
As we only add in and remove components of $\Delta(G, H)$, as opposed to individual edges, we always color all edges in a component in the same color. 

% \noindent $\bullet$ Cost function shows why we can ignore graphs with none of $S_G$ and $S_H$. 

Let us consider the cost function 
\[
w_{S_H, S_G}(S)=\begin{cases}
\frac{1}{|S_H|} & \text{ if } S\in S_H,\\
\frac{1}{|S_G|} & \text{ if } S\in S_G \text{, and}\\
0 & \text{ otherwise.}
\end{cases}
\]
As we only consider the face $C_{G, H}$, where all vertices $c_D\in C_{G, H}$ are constant on components of $\Delta(G, H)$, it is direct that $\langle w_{S_H, S_G}, c_D\rangle$ assumes one of the the values $0$, $1$, or $2$. 
It is direct that $\langle w_{S_H, S_G}, c_G\rangle=\langle w_{S_H, S_G}, c_H\rangle=1$. 
We claim we cannot have $\langle w_{S_H, S_G}, c_D\rangle=2$, which we will now prove.
% ---
Assume $\{S_G, S_H\}$ are not compatible but $\mathcal{S}\supseteq S_G,S_H$ is. 
If $S_H,S_G$ are not compatible because they form a clique not in either $G$ or $H$, then this is true also for any graph $F$ induced by a chordially compatible set $\mathcal{S}\supseteq S_G,S_H$. 
Therefore we must have that $\{S_H,S_G\}$ are not compatible because they create a longer cycle. 
We claim that either we have a chord present in both $G$ and $H$, or we induce a triangle not present in either $G$ or $H$. 
Thus consider a chord $gh$ in $F$, the graph induced by $c_\mathcal{S}$, that is not present in the graph induced by $c_{\{S_G, S_H\}}$. 
If $gh$ is present in both $G$ and $H$, then $\{g, h\}$ is not in $\Delta(G, H)$ and thus the edge is present in any graph induced by some subset of $\mathcal{M}$. % we must also have $gh$ in the graph induced by 
Otherwise, as $F$ is chordal, $gh$ is part of at least two triangles in the cycle. 
If one of these triangles is present in $G$ and one present in $H$, then $gh$ is present in both $G$ and $H$, hence both of these triangles are only present in one of $G$ or $H$. 
Repeating this argument we have that the whole triangulation of the cycle is present in only one of $G$ or $H$. 
Crucially, this triangulation will be connected in $\Delta(G, H)$. 
As it is not present in the graph induced by $c_{\{S_G, S_H\}}$, it must be on the form $\delta_H(x)$ (or $\delta_G(x)$) for some $x$. 
However, then $\ne_H(x)$ is complete in $G$ and in $H$, and it contains all nodes in the cycle. 
Therefore the cycle is a triangle, and this reduces to the case where $\{S_G, S_H\}$ induces a clique not present in $G$ or $H$. 

% then $F$ must contain a triangulation of this cycle,
% and by Lemma X.X there will be a triangle in $F$ that is not present in neither $G$ nor $H$.  
We conclude that we cannot have $\langle w_{S_H, S_G}, c_D\rangle=2$ and therefore $w_{S_H, S_G}$ determines a face containing $c_G$ and $c_H$, namely all $c_\mathcal{S}$ such that $S_H$ or $S_G$ is in $\mathcal{S}$, but not both. 
By \cref{prop: face of rhombus criterion} it suffices to consider this face for the reminder of this proof. 
Note that any $c_\mathcal{S}$ maximizing $w_{S_H, S_G}$ either induce all the purple, or all the green edges. 

We will now go through the remaining edges of $G$ and $H$ and color them green or purple as we simultaneously define a cost function forcing them to be together with $S_H$ or $S_H$ respectively. 

Let $gh$ be an uncolored edge in $H$ that is not in $G$, with $g\in O_H, h\in K_H$. 
Since $gh$ is not colored yet, as motivated in Case IIIb, we conclude that $\ne_H(g)\cap (K_H\setminus K_G)=\{h\}$.
We consider the cost function
\[
w_{\delta_H(g), S_G}(S)=\begin{cases}
\frac{1}{|\delta_H(g)|} & \text{ if } S\in \delta_H(g),\\
\frac{1}{|S_G|} & \text{ if } S\in S_G \text{, and}\\
0 & \text{ otherwise.}
\end{cases}
\]
Similar to the above we wish to show that either we have witnesses, or there cannot be a graph $F$ such that $\langle w_{\delta_H(g), S_G}, c_F\rangle =2$. 
Using \cref{thm: edge of face} and \cref{prop: face of rhombus criterion} we can then only consider all graphs $F$ such that $\langle w_{\delta_H(g), S_G}, c_F\rangle =1$. 
Therefore we can effectively color all edges in $\delta_H(g)$ purple. 
Repeating this for all uncolored edges will show the claim. 
See \cref{fig: cases long proof} for a schematic over the cases below. 
Note that we cannot have $\ne_G(g)=\ne_H(g)$ as the edge $gh$ is in $H$ but not in $G$.

% --- {\color{red} Here we being the cases, Jag har skrivit rätt fort så det behövs nog korigeras.}

\begin{enumerate}[label=(\roman*)]
    \item {} $g\in O_G\cap O_H$ and no colored edges to $g$. As described before, since all edges are uncolored we have that $\ne_G(g)$ is complete in $H$ and vice versa with $\ne_H(g)$. Therefore the conditions of \cref{lem: complete neigh} are satisfied as long as $\ne_G(g)\neq\ne_H(g)$ and thus we have witnesses.  % then all edges are already colored, making this a non-case. 
    \item {} $g\in O_G\cap O_H$ and all edges in $E(\delta_G(g))$ are colored green. If $\ne_G(h)\subseteq \ne_H(g)$, then $G\cup E(\delta_H(g))$ and $H\backslash E(\delta_H(g))$ are witnesses. If on the other hand $\ne_G(h)\backslash \ne_H(g)\neq \emptyset$, then $gh$ cannot be added to the green edges without creating a new cycle not present in neither $G$ nor $H$. Therefore we can use the exact same motivation using $w_{\delta_H(g), S_G}$ as we did about $w_{S_H, S_G}$, note that the reasoning is completely analogous.  
    Thus we use a cost function to narrow to a face of the polytope where $S_G$ and $\delta_H(g)$ are not both in any graph and we can thus color the edges of $E(\delta_H(g))$ purple.
    \item {} $g\in K_G\cap O_H$ and $\ne_G(h)\backslash K_H=\emptyset$. The, almost, verbatim argument made in (ii) can be applied in this case as well. 
    % As in case (ii), if $\ne_G(h)\subseteq \ne_H(g)$, then $G\cup \delta_H(g)$ and $H\backslash \delta_H(g)$ are witnesses. If on the other hand $\ne_G(h)\backslash \ne_H(g)\neq \emptyset$, then $gh$ cannot be added to the green edges without creating a new cycle not present in neither $G$ nor $H$. Thus we use a cost function as described above to narrow to a face of the polytope where $S_G$ and $\delta_H(g)$ are not both in any graph and we can thus color the edges of $\delta_H(g)$ purple.--- {\color{red} Jag fick exakt samma resonemang som i (ii) så de kan kanske slås ihop?}
    \item {} $g\in K_G\cap O_H$ and there exist a green edge $xh, x\in K_G\backslash K_H$. Again, $gh$ cannot be added to the green edges without creating a new cycle not present in neither $G$ nor $H$. Thus we use a cost function as described above to narrow to a face of the polytope where $S_G$ and $\delta_H(g)$ are not both in any graph and we can thus color the edges of $\delta_H(g)$ purple.
\end{enumerate}

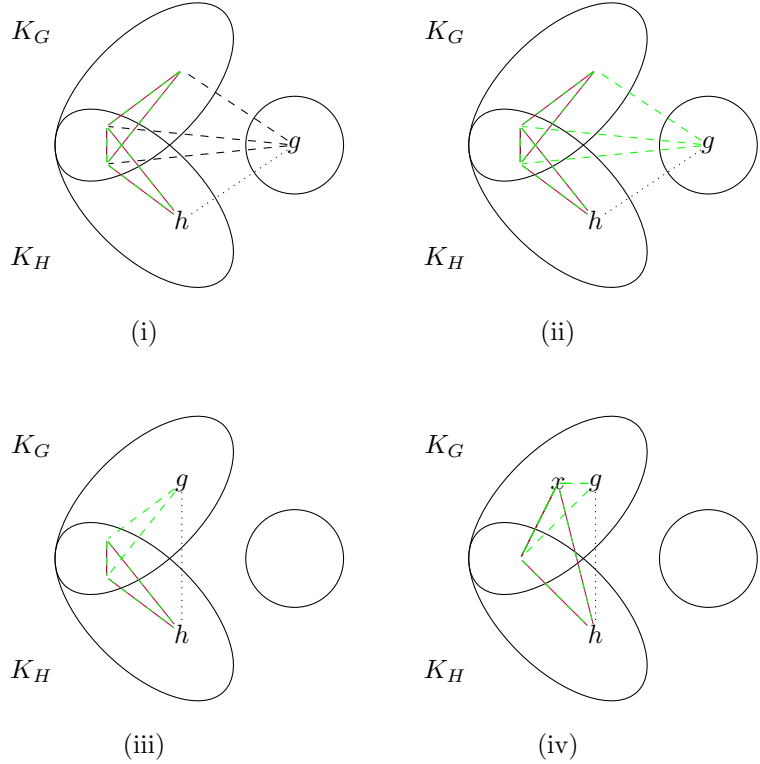
\begin{figure}
\centering
\begin{tikzpicture}[every node/.style={inner sep = 0pt}, scale = 0.5]

\begin{scope}
\draw[rotate = 45] (1,1) ellipse (3cm and 1.5cm);
\draw[rotate = -45] (1,-1) ellipse (3cm and 1.5cm);
\draw (4,0) circle (1.3cm);

% \draw (-5,-5) grid (5,5);
% \node at (0,0) {\color{green} 0};
% \node at (1,0) {\color{blue} 0};

\node (g) at (4,0) {$g$};
\node (h) at (1, -2) {$h$};
\node at (-3, 3) {$K_G$};
\node at (-3, -3) {$K_H$};
\node at (0, -5) {(i)};

\node (x1) at (-1,-0.5) {};
\node (x2) at (1,2) {};
\node (x3) at (-1,0.5) {};

\foreach \from/\to in {g/x1, g/x2, g/x3}{
  \draw[dashed] (\from) -- (\to);
}
\draw[dotted] (g) -- (h);
\foreach \from/\to in {x1/h, x3/h, x1/x3, x1/x2, x2/x3, x1/x3}{
  \draw[purple] (\from) -- (\to);
}
\foreach \from/\to in {x1/h, x3/h, x1/x3, x1/x2, x2/x3, x1/x3}{
  \draw[green, dashed] (\from) -- (\to);
}
% \draw[dashed, green] (x1) -- (x3);
\end{scope}

%%%%%%%%%%%%%%%%%%%%%%%%%%%%%%%%%%%%%%%%%%%%%%%%%%%%%%%%%%%%%%
\begin{scope}[shift={(11,0)}]
\draw[rotate = 45] (1,1) ellipse (3cm and 1.5cm);
\draw[rotate = -45] (1,-1) ellipse (3cm and 1.5cm);
\draw (4,0) circle (1.3cm);

\node (g) at (4,0) {$g$};
\node (h) at (1, -2) {$h$};
\node at (-3, 3) {$K_G$};
\node at (-3, -3) {$K_H$};
\node at (0, -5) {(ii)};

\node (x1) at (-1,-0.5) {};
\node (x2) at (1,2) {};
\node (x3) at (-1,0.5) {};

\foreach \from/\to in {g/x1, g/x2, g/x3}{
  \draw[dashed, green] (\from) -- (\to);
}
\draw[dotted] (g) -- (h);
\foreach \from/\to in {x1/h, x3/h, x1/x3, x1/x2, x2/x3, x1/x3}{
  \draw[purple] (\from) -- (\to);
}
\foreach \from/\to in {x1/h, x3/h, x1/x3, x1/x2, x2/x3, x1/x3}{
  \draw[green, dashed] (\from) -- (\to);
}
% \draw[dashed, green] (x1) -- (x3);
\end{scope}

%%%%%%%%%%%%%%%%%%%%%%%%%%%%%%%%%%%%%%%%%%%%%%%%%%%%%%%%%%%%%%
\begin{scope}[shift={(0,-11)}]
\draw[rotate = 45] (1,1) ellipse (3cm and 1.5cm);
\draw[rotate = -45] (1,-1) ellipse (3cm and 1.5cm);
\draw (4,0) circle (1.3cm);

\node (g) at (1,2) {$g$};
\node (h) at (1, -2) {$h$};
\node at (-3, 3) {$K_G$};
\node at (-3, -3) {$K_H$};
\node at (0, -5) {(iii)};

\node (x1) at (-1,0.5) {};
% \node (x2) at (0,2) {};
\node (x3) at (-1,-0.5) {};

\foreach \from/\to in {g/x1, g/x3}{
  \draw[dashed, green] (\from) -- (\to);
}
\draw[dotted] (g) -- (h);
\foreach \from/\to in {x1/h, x3/h, x1/x3, x1/x3}{
  \draw[purple] (\from) -- (\to);
}
\foreach \from/\to in {x1/h, x3/h, x1/x3, x1/x3}{
  \draw[green, dashed] (\from) -- (\to);
}
% \draw[dashed, green] (x1) -- (x3);
\end{scope}

%%%%%%%%%%%%%%%%%%%%%%%%%%%%%%%%%%%%%%%%%%%%%%%%%%%%%%%%%%%%%%
\begin{scope}[shift={(11,-11)}]
\draw[rotate = 45] (1,1) ellipse (3cm and 1.5cm);
\draw[rotate = -45] (1,-1) ellipse (3cm and 1.5cm);
\draw (4,0) circle (1.3cm);

\node (g) at (1,2) {$g$};
\node (h) at (1, -2) {$h$};
\node at (-3, 3) {$K_G$};
\node at (-3, -3) {$K_H$};
\node at (0, -5) {(iv)};

\node (x1) at (0,2) {$x$};
\node (x2) at (0,2) {};
\node (x3) at (-1,0) {};

\foreach \from/\to in {g/x1, g/x2, g/x3}{
  \draw[dashed, green] (\from) -- (\to);
}
\draw[dotted] (g) -- (h);
\foreach \from/\to in {x1/h, x3/h, x1/x3, x1/x2, x2/x3, x1/x3}{
  \draw[purple] (\from) -- (\to);
}
\foreach \from/\to in {x1/h, x3/h, x1/x3, x1/x2, x2/x3, x1/x3}{
  \draw[green, dashed] (\from) -- (\to);
}
% \draw[dashed, green] (x1) -- (x3);
\end{scope}

\end{tikzpicture}
\caption{The final 4 cases in the proof of \cref{thm:split edges}.}
\label{fig: cases long proof}
\end{figure}

Note that (iv) is in the situation for whenever $|\ne_G(h)|\ge 2$.
Symetrically, We have the same cases for every uncolored edge $gh$ in $G$ that is not in $H$, with $h\in O_G, g\in K_G$. We might have to repeat this recursively to color all edges that can be colored.

%If $\ne_G(h)=\emptyset$, then we can form witnesses $G'$ where we add all edges from $g$ to $\ne_H(h)$ and $H'$ where we remove all edges between $g$ and $\ne_H(g)$. {\color{red} This is \cref{lem: complete neigh} right?} 

% ***Cost function shows why we can focus only on graphs that contain exactly one of $S_G$ and $gh$. In the later situation we know it also contains $S_H$.

%We will begin to discuss the case when we have a green colored edge connected to $h$, say $xh$. 
%If $h\in K_H\setminus K_G$, we consider

% Note that if we have $\langle w_{\delta_H(g), S_G}, c_F\rangle = 2$, we must have all green edges in $F$, as well as the edge $gh$. 
% However, this would mean we have $c_F(\{g,h,x\})=1$, and thus $hx\in\delta_H($

%If $|\ne_G(h)|\geq 2$, then $\ne_H(g)$ contains at least one node $x\in K_H\cap K_G$. 
%---

%Now we can color $gh$ purple because it must be with $S_H$. Note that this happens in particular if $|\ne_G(h)|\ge 2$. We do this also to all such edges and to all uncolored edges in $G$ that are not in $H$.

The only uncolored $gh$ edges now are those that have exactly one other uncolored neighbor at each vertex, and no green neighbor at $h$ or purple neighbor at $g$.
The uncolored edges are therefore a union of cycles, let $C$ be one. 
If there is such a cycle we can then finish this case and construct witnesses as with the cycle at the end of Case IIIb. 
If there are no uncolored edges, we have, by a linear combination of the created cost functions, a cost function that shows $c_G,c_H$ forming an edge of the chordal polytope. 
\end{proof}

Importantly, the above shows that for large $n$, most pairs of vertices of $\CGP_n$ fulfill the rhombus criterion. 
We will utilize this later in \cref{sec: computations} to compute the edge structure of this polytope. 

\begin{remark} It would be desirable to give an explicit description of when $c_G,c_H$ are neighbors in $\CGP_n$ for split graphs $G,H$. From the long proof of \cref{thm:split edges} we can deduce that it happens
\begin{enumerate}
    \item if there exists a unique $x$ such that $x\in O_G\cup O_H$ with $\ne_G(x)\neq\ne_H(x)$ and both complete in both graphs (Case I and Case IIIa), 
    \item if $K_G\subsetneq K_H$ and $\ne_G(x)=\ne_H(x)$ for all $x\in O_H$ (Case IIa), 
    \item if $K_G\subsetneq K_H$ and $\forall x\in O_H,$ such that $ \ne_G(x)\neq\ne_H(x)$, $\ne_H(x)$ is not complete in $G$ (Case IIc), 
    \item if $\{S_G,S_H\}$ is not a compatible set and the coloring of edges in Case IIIc (i)-(iv) results in all edges colored either green or purple.
\end{enumerate}
The last condition is not as explicit as one would hope and it would be interesting if it could be formulated more succinctly.
\end{remark}

\begin{remark}
An alternative characterization of chordal graphs is that the maximal cliques form a so called junction tree (also known as clique tree), with edges for nontrivial intersections. The split graphs are then characterized by the junction tree being a star graph. One approach to try to generalize \cref{thm:split edges} to all chordal graphs could be to understand the polytopal edges in terms of the junction tree (star) and then generalize it to all junction trees. It would be interesting if this could lead to a proof of \cref{conj:chordal}.
\end{remark}
%%%%%%%%%%%%%%%%%%%%%%%%%%%%%%%%%%%%%%%%%%%%%
%---SUBSECTION: Tree CIM
\subsection{The tree \texorpdfstring{$\CIM$}{CIM} polytope}
We will end this section with a non-trivial example of a 0/1-polytope that does not fulfill the rhombus criterion. 
We will however begin to show that many non-edges do fulfill the rhombus criterion, therefore it is possible to compute the edge-structure using the algorithms presented in \cref{sec: computations}. 

As \cite{RP87} and \cite{LRS22} considered all DAGs such that the skeleton are trees, we will consider the following polytope
\[
\CIMT_n\coloneqq \conv\left(c_\GG\colon \GG\text{ is a DAG with skeleton a tree} \right). 
\]

\begin{proposition}\label{prop:STP is face}
The spanning tree polytope $T_n$ is a face of $\CIMT_n$. 
\end{proposition}

\begin{proof}
Notice that $\CIMT_n$ is the convex hull off a subset of all vertices of $\CIM_n$ containing the vertices of $\CIM_G$ whenever $G$ is a tree. 
Therefore any cost function maximizing in $\CIM_G$ over $\CIM_n$ will also maximize in $\CIM_G$ over $\CIMT_n$. 
The rest follows from the fact that $\CIM_G$ is a face of $\CIM_n$ \cite[Corollary 2.5]{LRS20}.
\end{proof}

\begin{proposition}\label{prop:CIMG is face}
Let $G$ be a tree. 
Then $\CIM_G$ is a face of $\CIMT_n$.
\end{proposition}

\begin{proof}
As for the spanning tree polytope; define the cost function 
\[
c(S)=\begin{cases}
0 & \text{ if } |S| = 2\\
-1 & \text{ otherwise.}
\end{cases}
\]
Let $\GG$ be a DAG maximizing $\langle c, c_\GG\rangle$ over $\CIMT_n$ with skeleton $G$. . 
As there exists DAGs $\DD$ with $\langle c, c_\DD\rangle=0$ and $c$ is a non-positive vector, we must have $\langle c, c_\GG\rangle=0$.
That is $c_\GG(S)=0$ for all $|S|\geq 3$. 
We note that we have $c_\GG(\{i,j\})=(v_G)_{ij}$ for every pair of vertices $i, j$. 
This gives us $c_\GG=v_G$ when we consider $v_G$ embedded in the same space as $c_\GG$ with the identification of $e_{\{i,j\}} = e_{ij}$. 
Thus the set of vertices maximizing $c$ over $\CIMT_n$ are exactly those that can be interpreted as vertices of $T_n$. 
The result follows. 
\end{proof}
%%%
\begin{proposition}
Let $G$ be an undirected tree and $\GG$ be a DAG with skeleton $G$ and no v-structures. 
The neighbors of $c_\GG$ in $\CIMT_n$ can be described as follows.
\begin{enumerate}
\item{ % A new edge, no new v-structure.
As in Spanning tree polytope. That is, one undirected edge is added and one edge is removed to break the created cycle.}
\item New v-structure(s), no new edge.
As in $\CIM_G$. That is, a vertex $i$ is getting two or more existing edges directed towards it creating v-structures at $i$.
\item New edge and new v-structure
For a vertex $i$ new edges (at least 1) are added, all directed towards $i$. As many edges as added are removed to break all created cycles. Existing edges in $G$ incident to $i$ that remain are given one of the two possible directions. If there is at least one v-structure at $i$ this is a neighbour.
Unless (triangle exception) exactly one new edge $j \to i$ is added with the edge $j  \ell$ removed and the edge $i  \ell$ is the only edge given direction towards $i$ among existing edges.
\end{enumerate}
\end{proposition}

Note that (2) can be considered a subcase of (3).
\begin{proof}
For the first case we know by Proposition \ref{prop:STP is face} that the edges in the undirected case are the same as the edges of the spanning tree polytope. 
Similarly, in the second case we have by Proposition \ref{prop:CIMG is face} that the neighbors of $\GG$ with identical underlying tree follows from the essential flips in $\CIM_G$ as proved in \cite{LRS22}.

To prove that the DAGs in case (3) are neighbors we define a cost function that is maximized only by $\GG$ and the proposed neighbor $\HH$. 
For a DAG $\HH$ satisfying (3) let $X$ be the vertices with an edge to $i$ in $\HH$ but not in $\GG$. Note that $|X|\ge 1$. Let $m=|\pa_\HH(i)|\ge 2$.
\[
w(S)=\begin{cases}
2 & \text{ if } S = \{j,k\} \text{ and } jk\in\GG\cap \HH \\
1 & \text{ if } S = \{j,k\} \text{ and } jk\in\GG \setminus \HH\\
-1 & \text{ if } S = \{j,k\} \text{ and } jk\notin\GG \\
2^{m}-m+2|X|-2 & \text{ if } S = \pa_\HH(i)\cup \{i\}\\
-1 & \text{ if } |S|\ge 3 \text{ and } i\in S\neq \pa_\HH(i)\cup \{i\}\\
-10 & \text{ if } |S|\ge 3 \text{ and } i\notin S\\
%0 & \text{ otherwise.}
\end{cases}
\]

The number of edges of the underlying tree is $n-1$, so $\langle w, c_\GG\rangle=2(n-1)-|X|$. Since the underlying graphs of $\GG$ and $\HH$ differ by $|X|$ edges and $c_\HH$ will have -1 for every set $S\cup\{i\}$, 
when $S\subseteq \pa_\HH(i)$ thus the term $2^m-(m-|X|)-2$ ensures that 
$\langle w, c_\HH\rangle$ is also $2(n-1)-|X|$. 
Any DAG $\DD\in \CIMT$ different from $\GG$ has either an edge not in $G$ or a v-structure which gives a negative contribution to the cost function. 
If $\langle w, c_\DD\rangle\ge\langle w, c_\GG\rangle$ we must thus have the positive weight $2^{m}-(m-|X|)-2$. 
If the common collider for this set in $\DD$ is $i$ we get 
$\pa_\DD(i)\subseteq \pa_\HH(i)$. In fact there must be equality since extra parents would give extra v-structures and thus negative weights. Since the edges in 
$\GG\cap \HH$ are given higher weight the only way to achieve maximal score is if $\DD$ is the same markov equivalence class as $\HH$.
If the common collider in $\DD$ is an other vertex and $m\ge 3$; then there will be at least one v-structure not including $i$ which lowers the total weight, 
that is $\langle w, c_\DD\rangle<\langle w, c_\HH\rangle$.

If $m=2$ we are in the situation of the triangle exception and the cost function will not be maximized by $\GG$ and $\HH$. In fact, for the triangle exception there exist two graphs as witnesses fulfilling the rhombus criteria, proving by Lemma \ref{lem: square criterion} that $\GG$ and $\HH$ do not form an edge in the polytope. %{\color{red} What has disappeared from here? This I think!}
Namely, the graphs that agree with $G$ except for the edges between $i,j,\ell$, in the following way. 
The first graph has $i \to \ell \leftarrow j$ and the second has the edges $j - i - \ell$. 
Recall that $G$ has $i - \ell - j$ and the graph in the triangle exception has $j \to i \leftarrow \ell$.

To prove that these are the only neighbors we must show that if 
$\HH$ has more than 2 colliders or one collider
$i$ but also other unrelated edges that differ from $\GG$.
Let $i$ be a collider in $\HH$ that has no other collider as ancestor. We will now form two witnesses for the rhombus criteria. 
To construct $\DD_1$ we start from $\GG$ and first direct and add the edges so $i$ gets the same parents as in $\HH$. To construct $\DD_2$ we start from $\HH$ and do the opposite operation, i.e. reversing some edges and removing some $e_1,\dots,e_r$, which will create some disconnected subtrees $T_1,\dots,T_r$. Since no ancestors of $i$ in $\HH$ were colliders we may further direct edges in those branches of the DAG without creating any v-structures. We must also add edges in $\DD_2$ to connect the currently disjoint subtrees $T_1,\dots,T_r$. 
The edge $e_1$ creates a cycle when added to $\GG$, let $f_1$ be the first edge (starting from $e_1$ moving away from $i$ towards $T_1$) of that cycle that does not belong to $T_1$. We now remove $f_1$ from $\DD_1$ and add it to $\DD_2$. In $\DD_2$ it will connect $T_1$ with the main component or with some other $T_j$ to form $T_j'$. The tree $T_1$ contains no collider in $\HH$ and we can thus direct $f_1$ towards $T_1$ and redirect the edges therein without changing the set of v-structures. 
Now we repeat this construction with $e_2,\dots,e_r$ possibly with $T_j$ enlarged to $T_j'$.
This construction gives two DAGs $\DD_1$ and $\DD_2$ with the joint multiset of edges the same as that of $\GG$ and $\HH$. Since $\DD_1$
will contain the v-structures of $\HH$ centered at $i$ and $\DD_2$ all the rest (if any), we can conclude that the rhombus criteria is satisfied.
\end{proof}

\begin{example}
\label{ex: cimt6 not square}
While one might hope that also $\CIMT_n$ would fulfill the rhombus criterion this is only true for $n\leq 5$. 
In \cref{fig: cimt not square} we give five essential graphs whose characteristic imsets spans a 3-dimensional face of $\CIMT_6$. 
A straightforward computation shows that the top two graphs does not form an edge of $\CIMT_6$ and it can be checked that they do not have witnesses.
In \cref{sec: computations} we explain how these examples can be found. 
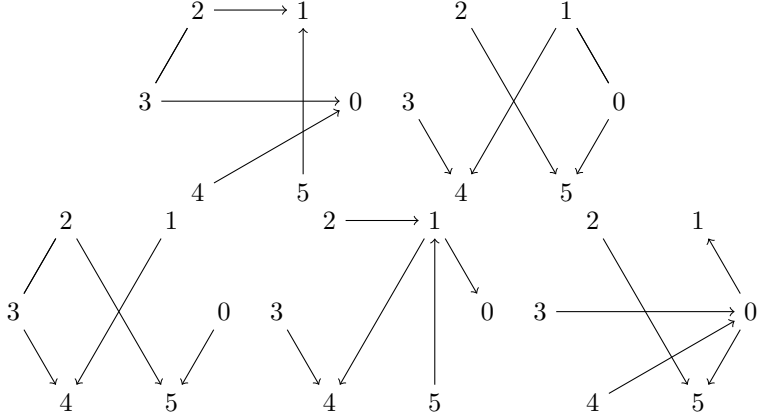
\begin{figure}
\[
\begin{tikzpicture}[scale=0.7]
\begin{scope}
\foreach \i in {0,...,5}{\node (p\i) at (60.0*\i:2) {$\i$};}
\draw[->] (p2) to (p1);
\draw (p2) to (p3);
\draw[->] (p3) to (p0);
\draw (p3) to (p2);
\draw[->] (p4) to (p0);
\draw[->] (p5) to (p1);
\end{scope}
\begin{scope}[shift={(5,0)}]
\foreach \i in {0,...,5}{\node (p\i) at (60.0*\i:2) {$\i$};}
\draw (p0) to (p1);
\draw[->] (p0) to (p5);
\draw (p1) to (p0);
\draw[->] (p1) to (p4);
\draw[->] (p2) to (p5);
\draw[->] (p3) to (p4);
\end{scope}

\begin{scope}[shift={(-2.5,-4)}]
\foreach \i in {0,...,5}{\node (p\i) at (60.0*\i:2) {$\i$};}
\draw[->] (p0) to (p5);
\draw[->] (p1) to (p4);
\draw (p2) to (p3);
\draw[->] (p2) to (p5);
\draw (p3) to (p2);
\draw[->] (p3) to (p4);
\end{scope}
\begin{scope}[shift={(2.5,-4)}]
\foreach \i in {0,...,5}{\node (p\i) at (60.0*\i:2) {$\i$};}
\draw[->] (p1) to (p0);
\draw[->] (p1) to (p4);
\draw[->] (p2) to (p1);
\draw[->] (p3) to (p4);
\draw[->] (p5) to (p1);
\end{scope}
\begin{scope}[shift={(7.5,-4)}]
\foreach \i in {0,...,5}{\node (p\i) at (60.0*\i:2) {$\i$};}
\draw[->] (p0) to (p1);
\draw[->] (p0) to (p5);
\draw[->] (p2) to (p5);
\draw[->] (p3) to (p0);
\draw[->] (p4) to (p0);
\end{scope}
\end{tikzpicture}
\]
%%%%%%%%%%%%%%%%%%%%%%%%%%%%%%%%%%%%%%%%%%%%%%%%
% THIS POLYTOPES IS DEFINED AS:
% Facets:
% 0: x3 >= 0
% 1: x1 + x3 - x5 >= 0
% 2: x1 - x3 + x5 >= 0
% 3: -x1 - x3 - x5 >= -2
% 4: -x1 >= -1
% 5: x5 >= 0

% Affine hull:
% 0: x2 = 0
% 1: x4 = 0
% 2: x6 = 0
% 3: x7 = 0
% 4: x1 + x8 = 1
% 5: x9 = 0
% 6: x10 = 0
% 7: x11 = 0
% 8: -x5 + x12 = 0
% 9: 1/2 x5 + 1/2 x12 + x13 = 1
% 10: x14 = 0
% 11: x15 = 0
% 12: x16 = 0
% 13: x17 = 0
% 14: x18 = 0
% 15: 1/3 x5 + 1/3 x12 - 1/3 x13 + x19 = 2/3
% 16: -1/4 x5 - 1/4 x12 + 1/4 x13 + 1/4 x19 + x20 = 1/2
% 17: 1/5 x5 + 1/5 x12 - 1/5 x13 - 1/5 x19 + 1/5 x20 + x21 = 3/5
% 18: x22 = 0
% 19: x23 = 0
% 20: x24 = 0
% 21: x25 = 0
% 22: x26 = 0
% 23: x3 + x27 = 1
% 24: -1/2 x3 + 1/2 x27 + x28 = 1/2
% 25: x29 = 0
% 26: 1/3 x3 - 1/3 x27 + 1/3 x28 + x30 = 2/3
% 27: 1/4 x3 - 1/4 x27 + 1/4 x28 - 1/4 x30 + x31 = 1/2
% 28: -1/5 x3 + 1/5 x27 - 1/5 x28 + 1/5 x30 + 1/5 x31 + x32 = 3/5
% 29: x33 = 0
% 30: x34 = 0
% 31: x35 = 0
% 32: x36 = 0
% 33: x37 = 0
% 34: x38 = 0
% 35: x39 = 0
% 36: x40 = 0
% 37: x41 = 0
% 38: x42 = 0
% 39: x43 = 0
% 40: x44 = 0
% 41: x45 = 0
% 42: x46 = 0
% 43: x47 = 0
% 44: x48 = 0
% 45: x49 = 0
% 46: x50 = 0
% 47: x51 = 0
% 48: x52 = 0
% 49: x53 = 0
% 50: x54 = 0
% 51: x55 = 0
% 52: x56 = 0
% 53: x57 = 0
%%%%%%%%%%%%%%%%%%%%%%%%%%%%%%%%%%%%%%%%%%%%%%%%%%%
\caption{Five graphs that spans a face of $\CIMT_6$. This face does not fulfill the rhombus criterion and hence $\CIMT_6$ does not either.}
\label{fig: cimt not square}
\end{figure}
\end{example}

\section{Examples of the rhombus criterion}
\label{sec: examples}
For a polytope to fulfill the rhombus criterion is indeed very restrictive when considering polytopes in general. 
However, when considering common representations of some polytopes it keeps on appearing. 
We have already seen examples in \cref{prop: STP is rhombus}, \cref{cor: Cube is rhombus}, \cref{thm: CIMG is rhombus} and \cref{prop: fixed order CIM is rhombus} of polytopes fulfilling the rhombus criterion. In this section, we will consider some other polytopes from the literature that also turn out to fulfill the rhombus criterion. 
We will begin with a very direct example, namely the \emph{cross polytope} $C_n\coloneqq\conv(\pm e_i\colon 1\leq i\leq n)$.

\begin{proposition}
The rhombus criterion is true for the cross polytope, $C_n$.
\end{proposition}

\begin{proof}
The only non-edges of $C_n$ are $e_i$, $-e_i$ for all $1\leq i\leq n$. 
Then we have for all $j\neq i$ that $e_i+(-e_i)=e_j+(-e_j)$. 
Hence $e_i$, $-e_i$ fulfills the rhombus criterion with $e_j$, $-e_j$ as witnesses. 
\end{proof}

Another, slightly harder, example is the \emph{Birkhoff polytope} $B_n$ that is commonly defined as the convex hull of all $n\times n$ permutation matrices. 

\begin{proposition}
The rhombus criterion is true for the Birkhoff polytope, $B_n$.
\end{proposition}

This also follows as a corollary from \cref{prop: k-assignment square}, we will however first present a more elementary proof that hopefully is easier for the reader. 

\begin{proof}
We know that $(P_\sigma, P_\omega)$ is an edge of $B_n$ if and only if $\sigma^{-1}\omega$ is a cycle. 
Thus assume this is not true, that is assume that $\sigma^{-1}\omega=\pi_1\pi_2$ for some disjoint permutations $\pi_1$ and $\pi_2$.
% Notice that we have $\pi_1\pi_2=\pi_2\pi_1$.
Then we wish to show that 
\[
P_\sigma+P_\omega = P_{\sigma\pi_1} + P_{\sigma\pi_2}. 
\]
Using our equality $\sigma^{-1}\omega=\pi_1\pi_2$, the above is equivalent to $I + P_{\pi_1\pi_2} = P_{\pi_1}+P_{\pi_2}$. 
However as $\pi_1$ and $\pi_2$ are disjoint this is the equality (up to some abuse of notation) 
\[
\begin{pmatrix} I&0\\0&I\end{pmatrix} + \begin{pmatrix} P_{\pi_1} &0\\0&P_{\pi_2} \end{pmatrix}
= \begin{pmatrix} P_{\pi_1}&0\\0&I\end{pmatrix} + \begin{pmatrix} I &0\\0&P_{\pi_2} \end{pmatrix}.
\]
\end{proof}

The $k$-assignment polytope $M_{m,n,k}$ was studied in \cite{GL09} and is defined as the convex hull over all 0/1-matrices of size $m\times n$ with exactly $k$ number of 1, such that no row or column has more than one 1. 
This polytope is a natural generalization of the Birkhoff polytope; when $m=n=k$ then $M_{m,n,k} = B_n$. 
Let us consider $k$-matchings in the complete graph $K_{m, n}$, that is subsets $\mathcal{M}$ of the edges in $K_{m,n}$ of size $k$. 
There is a natural bijection between the vertices of $M_{m,n,k}$ and these matchings as $ij\in \mathcal{M}$ if and only if $\alpha_{i,j} =1$, where $\alpha$ is a vertex of $M_{m,n,k}$. 
Then we will say that a path $P$ is \emph{alternating} (with respect to the two matchings $\mathcal{M}_1$ and $\mathcal{M}_2$) if every other edge of $P$ lies in $\mathcal{M}_1$ and the rest in $\mathcal{M}_2$. 
% Note that our use of "alternating" differs slightly from some litterature where additional requirements can be found, for example the first edge in $P$ being unmatched. 
Then we can restate Theorem 13 in \cite{GL09} as follows. 

\begin{theorem} \cite{GL09} 
Let $\alpha$ and $\beta$ be vertices of $k$-assignment polytope $M_{m,n,k}$ and $\mathcal{M}_\alpha$, $\mathcal{M}_\beta$ the corresponding $k$-matchings. 
Then the pair $\alpha$, $\beta$ is an edge of $M_{m,n,k}$ if and only if the $(\mathcal{M}_\alpha\cup \mathcal{M}_\beta)\setminus(\mathcal{M}_\alpha\cap \mathcal{M}_\beta)$ is the disjoint union of at most two alternating paths whose union contains the same number of edges from $\mathcal{M}_\alpha$ and $\mathcal{M}_\beta$.
\end{theorem}

\begin{proposition}
\label{prop: k-assignment square}
The rhombus criterion is true for the k-assignment polytopes, $M_{m,n,k}$.
\end{proposition}

\begin{proof}
Let $\alpha$, $\beta$ be a non-edge of $M_{m,n,k}$ and let $\mathcal{M}_\alpha$ and $\mathcal{M}_\beta$ be the corresponding matches. 

We note that any vertex has at most one edge from $\mathcal{M}_\alpha$ and one from $\mathcal{M}_\beta$.
It follows that every vertex in $(\mathcal{M}_\alpha\cup \mathcal{M}_\beta)\setminus(\mathcal{M}_\alpha\cap \mathcal{M}_\beta)$ has degree at most 2, and thus $(\mathcal{M}_\alpha\cup \mathcal{M}_\beta)\setminus(\mathcal{M}_\alpha\cap \mathcal{M}_\beta)$ is a union of (vertex) disjoint alternating paths and cycles. 
Moreover we note that $\mathcal{M}_\alpha\setminus \mathcal{M}_\beta=\mathcal{M}_\alpha\setminus(\mathcal{M}_\alpha\cap \mathcal{M}_\beta)$ is equinumerous to $\mathcal{M}_\beta\setminus\mathcal{M}_\alpha=\mathcal{M}_\beta\setminus(\mathcal{M}_\alpha\cap \mathcal{M}_\beta)$ as both $\mathcal{M}_\alpha$ and $\mathcal{M}_\beta$ contains exactly $k$ edges. 

As every path in $(\mathcal{M}_\alpha\cup \mathcal{M}_\beta)\setminus(\mathcal{M}_\alpha\cap \mathcal{M}_\beta)$ is alternating, we must have that $|P\cap (\mathcal{M}_\alpha\setminus\mathcal{M}_\beta)|$ and $|P\cap (\mathcal{M}_\beta\setminus \mathcal{M}_\alpha)|$ differ with at most one. 
Then if we sum over all inclusion maximal paths (and cycles) in $(\mathcal{M}_\alpha\cup \mathcal{M}_\beta)\setminus(\mathcal{M}_\alpha\cap \mathcal{M}_\beta)$  we have 
\begin{align*}
\sum_P |P\cap (\mathcal{M}_\alpha\setminus\mathcal{M}_\beta)| &- |P\cap (\mathcal{M}_\beta\setminus \mathcal{M}_\alpha)|\\
&=\sum_P |P\cap (\mathcal{M}_\alpha\setminus\mathcal{M}_\beta)| - \sum_P |P\cap (\mathcal{M}_\beta\setminus \mathcal{M}_\alpha)|\\
&= |\mathcal{M}_\alpha\setminus\mathcal{M}_\beta| - |\mathcal{M}_\beta\setminus\mathcal{M}_\alpha|=0
\end{align*}
$(\mathcal{M}_\alpha\cup \mathcal{M}_\beta)\setminus(\mathcal{M}_\alpha\cap \mathcal{M}_\beta)$ is a union of (vertex) disjoint alternating paths and cycles. 
However, as the summands in the first sum only assumes the values $1$, $0$, and $-1$, there is either a path $P_1$ such that $|P_1\cap (\mathcal{M}_\alpha\setminus\mathcal{M}_\beta)| = |P_1\cap (\mathcal{M}_\beta\setminus \mathcal{M}_\alpha)|$, then let $P_2=\emptyset$, or two paths, $P_1$ and $P_2$ such that $P_1\cup P_2$ has an equal amount of edges from $\mathcal{M}_\alpha$ and $\mathcal{M}_\beta$. 

Then define $\mathcal{M}_{\alpha'}$ from $\mathcal{M}_\alpha$ via exchanging all edges $(P_1\cup P_2)\cap \mathcal{M}_\alpha$ with $(P_1\cup P_2)\cap \mathcal{M}_\beta$, and similarly with  $\mathcal{M}_{\beta'}$ from $\mathcal{M}_\beta$. 
Let $\alpha'$ and $\beta'$ denote the vectors corresponding to $\mathcal{M}_{\alpha'}$ and $\mathcal{M}_{\beta'}$. 
We need to show is that $\alpha'$ and $\beta'$ are vertices of $M_{m,n,k}$, or equivalently that $\mathcal{M}_{\alpha'}$ and $\mathcal{M}_{\beta'}$ are $k$-matchings, that $\alpha+\beta=\alpha'+\beta'$, and that $\alpha'\neq \beta$.
That $\alpha+\beta=\alpha'+\beta'$ is direct from construction. 
If $\alpha'=\beta$, then $\alpha$, $\beta$ is an edge of $M_{m,n,k}$, which is not true by assumption. 
Hence what is left is that $\mathcal{M}_{\alpha'}$ and $\mathcal{M}_{\beta'}$ are $k$-matchings. 

First we see that $\mathcal{M}_{\alpha'}$ contains exactly $k$ edges as we removed equally many edges as we added in. 
For the sake of contradiction, assume $\mathcal{M}_{\alpha'}$ is not a matching. 
Since $P_1$ and $P_2$ were alternating, this mean we added in an edge $ij$ in $P_1$ or $P_2$ such that there is an edge $i'j$ in $\mathcal{M}_{\alpha'}\setminus (P_1\cup P_2)=\mathcal{M}_{\alpha}\setminus (P_1\cup P_2)$. 
As $P_1$ and $P_2$ were inclusion maximal in $(\mathcal{M}_\alpha\cup \mathcal{M}_\beta)\setminus(\mathcal{M}_\alpha\cap \mathcal{M}_\beta)$ and by definition $\mathcal{M}_{\alpha'}\subseteq (\mathcal{M}_\alpha\cup \mathcal{M}_\beta)\setminus(\mathcal{M}_\alpha\cap \mathcal{M}_\beta)$ we must have $i'j\in \mathcal{M}_\alpha\cap \mathcal{M}_\beta$.
This however contradicts that both $\mathcal{M}_\alpha$ and $\mathcal{M}_\beta$ were matchings. 
Thus the result follows. 
\end{proof}

Another well studied family of polytopes are the \emph{stable set polytopes}.
Let $G$ be a graph and $A\subseteq [n]$ a subset of all nodes. 
Then $A$ is \emph{stable} if $G|_A$ does not contain edges. 
For any subset $A\subseteq [n]$ the \emph{incidence vector} of $A$ is defined as
\[
\chi_A(i)=\begin{cases}
1&\text{ if $i\in A$, and }\\
0&\text{ otherwise.}\\
\end{cases}
\]
Then we define the stable set polytope of $G$, as the convex hull of all incidence vectors of stable sets, that is
\[
\STAB(G)\coloneqq \conv\left(\chi_A\colon A \text{ is stable in }G\right). 
\]
Due to Chv\'atal we have a characterization of the edges of $\STAB(G)$.
%---THEOREM: Chvatal Characterization----
\begin{theorem}
\cite[Theorem 6.2]{C75}
\label{thm: chvatal characterization}
Let $G = ([n],E)$ be an undirected graph, let $A$ and $B$ stable sets and $\chi_A,\chi_B\in\RR^n$ be two vertices of $\STAB(G)$.  
Then $\chi_A$, $\chi_B$ is an edge of $\STAB(G)$ if and only if the subgraph of $G$ induced by $A\setminus B\cup B\setminus A$ is connected. 
\end{theorem}
We can again rephrase this in terms of the rhombus criterion. 
\begin{proposition}
For any undirected graph $G = ([p],E)$ the polytope $\STAB(G)$ fulfills the rhombus criterion.  
\end{proposition}

\begin{proof}
Let $\chi_A$, $\chi_B$ be a non-edge of $\STAB(G)$, that is subgraph of $G$ induced by $A\setminus B\cup B\setminus A$ is not connected. 
Then consider a inclusion maximal connected component, $C$, of $A\setminus B\cup B\setminus A$. 
Define $A' = A\setminus C\cup \left(B\cap C\right)$ and $B' = B\setminus C\cup \left(A\cap C\right)$. 
It follows from definition that $\chi_A+ \chi_B=\chi_{A'}+\chi_{B'}$ and $A'$ is not equal to either $A$ nor $B$.
Thus what is left to show is that $A'$ and $B'$ are stable sets. 
Therefore assume that $A'$ is not a stable set, that is there are vertices $i, j\in A'$ such that $ij\in G$. 
As $A' = A\setminus C\cup \left(B\cap C\right)$ we have tree cases, $i,j\in A\setminus C$, $i,j\in B\cap C$, and $i\in A\setminus C$ and $j\in B\cap C$. 
The first two cases cannot be true as both $A$ and $B$ were stable, leaving us with $i\in A\setminus C$ and $j\in B\cap C$. 
Again since $B$ was stable we must have $i\in A\setminus \left(C\cup B\right)$ and thus $i\in A\setminus B$. 
However, as $C$ was  an inclusion maximal connected component, and $ij\in A\setminus B\cup B\setminus A$ we must have $ij\in C$, a contradiction. 
\end{proof}

%%%%%%%%%%%%%%%%%%%%%%%%%%%%%%%%%%%%%%%%%%%%%
%---SECTION: The algorithm 
\section{Efficient Computations of the Skeleta}
\label{sec: computations}
% Verifying that $\alpha$, $\beta$ is an edge can be done comparatively quickly; verifying that $\alpha$, $\beta$ is a non-edge is hard (REFERENS?) . 
Checking whether $\alpha$, $\beta$ is an edge or not can be hard as it is effectively a linear optimization problem (see \cref{prop: proj of edge}), which are known to be hard \cite{DS18}. 
In practice however, it done rather efficient due to highly optimized solvers \cite{H17}.
In general to show that $\alpha$, $\beta$ is a non-edge we need vertices $\{\gamma_i\}_{i=1}^d\subseteq V\setminus\{\alpha, \beta\}$, where $d$ is the dimension of $P$, such that
\begin{equation}
\label{eq: non-edge}
t\alpha + (1-t)\beta =\sum_i x_i\gamma_i
\end{equation}
for some non-negative real numbers $t$ and $\{x_i\}$.

For several of the polytopes discussed in this paper a direct computation of the edges is not feasable.
This is however mainly due to the dimension that can, in the worst case, grow exponentially, for example $\dim \CIMT_n = 2^n-n-2$. 
Many methods rely on finding the facets of the polytope, and from this deduce the edge structure. 
This does however, in theory, give access to all faces; we compute significantly more information than is necessary. 
Therefore algorithms for doing this are known to have a worst case of $v^{\left\lfloor \frac{d}2\right\rfloor}$, where $v$ is the number of vertices and $d$ is the dimension of the polytope \cite{Zie95}. 
The list of edges are however significantly smaller, at most $\binom{v}{2}$. 
In this section we will discuss how we can utilize the rhombus criterion to compute the edge-list of some 0/1-polytopes. 
This allows us to compute the full edge-list of several additional polytopes, for example $\CIM_5$ and $\CIMT_6$. 
This despite $\CIMT_6$ not fulfilling the rhombus criterion (see \cref{ex: cimt6 not square}).  
Let us first state the goal. 

\begin{goal}
Let $P$ be a $0/1$-polytope with vertices $V$. % , and let $\alpha,\ \beta\in V$.
Construct an algorithm to determine all pairs $\alpha,\ \beta\in V$ that are edges of $P$. 
\end{goal}
Implicit in this goal is the algorithm needs to be computationally feasible, even for larger dimensions. 
For example, $\CIM_5$ has $8782$ vertices, something a standard software like \texttt{cdd} \cite{cddlib} or \texttt{ppl} \cite{ppl06} usually can handle, but is 26-dimensional, and it will therefore take a significant time to find the edge structure. 
Our main goal is to obtain a better scaling in $d$, the dimension. 

To this end we will assume we have a list $\mathcal{L}$ of unordered pairs of vertices where an entry $(\{\alpha, \beta\},  e_{\alpha, \beta})\in\mathcal{L}$ stores whether we know that $\conv(\alpha, \beta)$ is not an edge of $P$, unsure whether $\conv(\alpha, \beta)$ is an edge of $P$, or are sure that $\conv(\alpha, \beta)$ is an edge of $P$, with $e_{\alpha, \beta}=-1$, $e_{\alpha, \beta}=0$, or $e_{\alpha, \beta}=1$, respectively. 
Notice that we can already have issues due to the size of $\mathcal{L}$. 
Just creating this list is in and by itself a non-trivial task. 
If such a list cannot be saved the following methods are unfortunately not applicable. 
Regardless, from this perspective all we need to do is determine $e_{\alpha, \beta}$ for all entries  in $\mathcal{L}$. 

%%%%%%%%%%%%%%%%%%%%%%%%%%%%%%%%%%%%%%%%%%%
%---SUBSECTION: Efficiency and optimization
\subsection{Efficiency and Optimizations}
\label{subsec: efficiency}
Checking every entry in $\mathcal{L}$ is a daunting task. 
This especially as computationally verifying whether a pair of vertices is an edge or not is non-trivial in higher dimensions.
Therefore we require some way of reducing the dimension. 

Let $P$ be a $0/1$-polytope with vertices $V$, and let $\alpha,\ \beta\in V$.
Then we define 
\[
\mathcal{F}_{\alpha,\beta}\coloneqq \left\{v\in V\colon \alpha_i=\beta_i\Rightarrow v_i=\alpha_i=\beta_i \right\}.
\]
That is, $\mathcal{F}_{\alpha,\beta}$ is the set of all vertices that coordinate-wise agree with $\alpha$ and $\beta$, whenever they agree. 
\begin{theorem}
Let $P$ be a $0/1$-polytope with vertices $V$, and let $\alpha,\ \beta\in V$.
Then $F_{\alpha, \beta}\coloneqq \conv\left(v\in\mathcal{F}_{\alpha, \beta}\right)$ is a face of $P$. 
\end{theorem}
\begin{proof}
It is enough to find a cost function $c$ such that $\mathcal{F}_{\alpha, \beta} = \argmax_{v\in V}\{c^Tv\}$, that is, a vertex $v$ maximizes $c^Tv$ if and only if $v\in \mathcal{F}_{\alpha, \beta}$.
Therefore define 
\[
c_i\coloneqq\begin{cases}
1 &\text{if } \alpha_i=\beta_i=1\\
-1 &\text{if } \alpha_i=\beta_i=0\\
0&\text{otherwise.}
\end{cases}
\]
Then to maximize $c$ over all $0/1$ vectors $u$ we need that $u_i=1$ whenever $\alpha_i=\beta_i=1$, $u_i=0$ whenever $\alpha_i=\beta_i=0$, and all other coordinates does not matter. 
As this is true for all $0/1$ vectors, it is especially true for all vectors in $V$. 
Then the result follows by definition of $\mathcal{F}_{\alpha, \beta}$. 
\end{proof}

In fact, if $P$ is the $d$-cube, then $F_{\alpha,\beta}$ is the smallest face containing both $\alpha$ and $\beta$. 
Therefore, for general $0/1$ polytopes we must consider the whole of $F_{\alpha,\beta}$ when determining if $\alpha, \beta$ is an edge of $P$. 
However, if the dimension of $F_{\alpha,\beta}$ is small enough, then we can apply any standard algorithm to determine the edge-structure of $F_{\alpha,\beta}$. 

This immediately gives us a first idea on how to tackle the problem at hand, for every $(\alpha, \beta, 0)\in \mathcal{L}$ we consider $F_{\alpha, \beta}$, compute the edge structure and update $\mathcal{L}$ accordingly. 
However, if $\alpha+\beta= \mathbf{1}$, that is $\alpha$ and $\beta$ does not agree on any coordinate, then $\mathcal{F}_{\alpha, \beta} = V$ and thus $F_{\alpha, \beta} = P$. 
Therefore, this idea is only feasible if we can rule out extreme cases similar to this. 
To this end however we have the rhombus criterion that can be efficiently checked in $\mathcal{L}$. 

In theory this is straightforward, consider every pair of entries $\{\alpha, \beta\}\times e_{\alpha, \beta},\ \{\alpha', \beta'\}\times e_{\alpha', \beta'}\in\mathcal{L}$.
If $\alpha+\beta=\alpha'+\beta'$ then \cref{lem: square criterion} tells us that neither $\alpha$, $\beta$ nor $\alpha'$, $\beta'$ is an edge of $P$, and hece $e_{\alpha,\beta}=e_{\alpha',\beta'}=-1$. 
A naive implementation of this will however take approximately $d\binom{v}{2}^2\in O(dv^4)$ time, but it can be done more efficiently. 

When creating $\mathcal{L}$ we can add the additional element $\gamma = \alpha +\beta$ to each entry. 
Then after sorting $\mathcal{L}$ after $\gamma$ (using any order) gives us a list where $\alpha$, $\beta$ fulfills the rhombus criterion if and only if either the entry above or below $(\{\alpha, \beta\}, e_{\alpha, \beta}, \gamma)$ are witnesses. 
Thus this can be done in $d\binom{v}{2}\log\binom{v}{2} +d\binom{v}{2}\in O\left(dv^2\log v\right)$ time. 
In practice this is very efficient due to highly optimized implementations, for example in the \texttt{python} library \texttt{pandas} \cite{pandas1, pandas2}.
See \cref{fig: computation time}  for computational time. 

If we have a theoretical guarantee that $P$ fulfills the rhombus criterion we can then draw the conclusion that for every other pair we have $e_{\alpha, \beta}=1$. 
This might however not be the case, especially when exploring whether a class might fulfill the rhombus criterion, and therefore we will need a way to verify whether a pair of vertices is an edge of $P$.

%%%%%%%%%%%%%%%%%%%%%%%%%%%%%%%%%%%%%%%%%%%
%---SUBSECTION: Edge-verification functions
\subsection{Edge verification functions}
\label{subsec: edge verification}
Here we will present two general ways of verifying whether a pair of vertices is an edge.  
This since one will scale better with dimension, but will have several deficits. 
The other being more reliable, but scales worse with dimension. 
As the main point of this method being a better scaling with dimension, we believe both being of significance. 
Similar to before it is enough to consider $F_{\alpha, \beta}$, as opposed to $P$.

To verify that $\alpha$, $\beta$ is an edge of $P$ it is enough to find a cost function $c$ that maximizes exactly in $\alpha$ and $\beta$ out of all vertices in $V$. 
To do this we can take any cost function $c$ such that $\langle c,\alpha\rangle = \langle c, \beta\rangle$. 
If there is a vertex $\mu\in V$ sucht that $\langle c,\mu\rangle > \langle c,\alpha\rangle$ we update $c$ such that $\langle c,\mu\rangle < \langle c,\alpha\rangle= \langle c, \beta\rangle$. 
If such a vertex does not exist, we are done. 
Left to do is how to update $c$. 

%---LEMMA: Cost Function Update
\begin{lemma}
\label{lem: cost function update}
Let $P$ be a polytope with vertex set $V$ and $\alpha,\ \beta\in V$. 
Assume $c$ is a cost function such that for some $\mu\in V$ we have $\langle c, \mu \rangle > \langle c, \alpha\rangle=\langle c, \beta\rangle$. 
Let $c_\mu$ be a face-defining cost-function for $\{\mu\}$.
Define $\delta = \alpha-\beta$ and $p = \left(c_\mu-\frac{\langle\delta, c_\mu\rangle}{\langle \delta, \delta\rangle}\delta\right)$. 
For some $\varepsilon$ with $|\varepsilon|$ small enough, we have
\[
c' = c-\left( \frac{\langle c, \alpha-\mu\rangle}{ \left\langle p, \alpha-\mu\right\rangle}+\varepsilon \right)p 
\]
we have $c'^T\mu < c'^T\alpha=c'^T\beta$.  
Moreover, the sign of $\varepsilon$ is the same as $\left\langle p, \alpha-\mu\right\rangle$.
\end{lemma}

\begin{proof}
Projecting $c_\mu$ onto the orthogonal space of $\spans(\delta)$ we get $p$. 
As $c$ already is in this space, so is $c_\lambda\coloneqq c-\lambda p$ for all $\lambda$. 
Then as $\langle c_\mu, \mu\rangle >\langle c_\mu, \alpha\rangle$ and $\langle c_\mu, \mu\rangle >\langle c_\mu, \beta\rangle$ we get $\langle p, \mu\rangle >\langle p, \alpha\rangle=\langle p, \beta\rangle$. 
This implies that $\langle c_\lambda, \mu - \alpha\rangle$ is a non-constant linear function in $\lambda$, and hence have a root. 
The rest follows by direct computations. 
\end{proof}

To obtain $c_\mu$ we can always choose $c_\mu = \mu - \frac{1}{|V|}\sum_{v\in V}v$, however for $0/1$-polytopes we can also choose $c_{\mu}=2\mu -\mathbf{1}$, which is computationally easier to deal with. 
Hence the only unknown above is $\varepsilon$. 
Computationally we need to avoid $\langle p, \alpha -\mu\rangle = 0$, as then we divide by $0$ in \cref{lem: cost function update}.
Therefore we might have to nudge $c_\mu$ a little, which can be done via adding a random vector $r$ with $\max \{ |r_i|\}$ small enough. 
If we further normalize $c_\mu$ we can choose $\max \{ |r_i|\} < \frac{1}{d}$. 
As for choosing $\varepsilon$, we can assume $c$ is normalized, and hence choosing a small constant as $\varepsilon$ is sufficient for our needs. 
This gives us \cref{alg: verify edge num} that utilizes \cref{lem: cost function update} and tries to verify that $\alpha, \beta$ is an edge via finding a cost function. 
%---ALGORITHM: Verify Edge---
\begin{algorithm}
  \caption{Verify Edge Numerical}
  \label{alg: verify edge num}
  \raggedright
  \hspace*{\algorithmicindent} \textbf{Input:} Two vertices $\alpha$ and $\beta$, and the vertex set $V$. Number of iterations $I$.\\
  \hspace*{\algorithmicindent} \textbf{Output:} TRUE if we can verify that $\conv(\alpha, \beta)$ is an edge of $P$ in $I$ or less iterations, FALSE otherwise.
  \begin{algorithmic}[1]
    \State $\delta \gets\alpha-\beta$
    \State $c \gets$ a random vector orthogonal to $\delta$
    % \State $c\gets c-\frac{\langle\delta, c\rangle}{\langle\delta, \delta\rangle}\delta$
    \For{$i\leq I$}
        \If{There is $\mu\in V$ such that $\langle c, \mu \rangle > \langle c, \alpha\rangle$}
            \State $c\gets c'$ (as in \cref{lem: cost function update})
        \Else
            \State Remove from $V$ all vertices not maximizing $c$.
        \EndIf
        \If{$|V|\leq 3$}
            \State \Return TRUE
        \EndIf
    \EndFor
    \State \Return FALSE
  \end{algorithmic}
\end{algorithm}

% \Cref{alg: verify edge num} moves in the orthogonal space of $\alpha-\beta$ as $\alpha, \beta$ is an edge of $P$ if and only if $\mathbf{0}$ is a vertex of the projection of $P$ onto the orthogonal space of $\alpha-\beta$. 
Notice that \cref{alg: verify edge num} can fail to verify an edge because it simply does not find a cost function, but never verifies a pair $\alpha, \beta$ that does not give us an edge of $P$. 
While we believe that, as the number of iterations tend to infinity ($I\to\infty$), the probability of \cref{alg: verify edge num} failing to verify an edge tends to $0$, such a result would not help computationally.
This as \cref{alg: verify edge num} is only more efficient than the algorithm presented below if we find the correct cost function early on. 
% We will discuss why later.  % would ultimately be irrelevant as we only save time if \cref{alg: verify edge num} exits within relatively few iterations.  
As the algorithms should be applicable for non-edges as well we present the second algorithm that builds on the following propostion. % rephrase the question of whether $\alpha,\ \beta$ is an edge of $P$. 
\begin{proposition}
\label{prop: proj of edge}
Let $P$ be a polytope and $\alpha$ and $\beta$ vertices of $P$. 
Define $\pi$ be the projection of $P$ onto the orthogonal space of $\spans(\alpha-\beta)$. 
Then $\alpha, \beta$ is an edge of $P$ if and only if $\pi(\alpha)$ (or $\pi(\beta)$) is a vertex of the projection of $P$. 
\end{proposition}

\begin{proof}
Denote $\delta = \alpha-\beta$, let $W$ be the orthogonal space to $\spans(\delta)$ and choose a basis of $W$ as $b_1,\dots, b_{d-1}$. 
Notice that $b_1,\dots, b_{d-1}, \delta$ is a basis of $\mathbb{R}^d$. 

Assume $\pi(\alpha)$ is a vertex of $\pi(P)$. 
Then there is a cost function $c'=(c_1,\dots,c_{d-1})$ in $W$ maximizing in $\pi(\alpha)$ over $\pi(P)$. 
Let $c=(c_1,\dots,c_{d-1}, 0)$. 
As $\langle c, v\rangle$ does not depend on $\spans(\delta)$ we have that $c$ maximize in $\left(\pi(\alpha)+\spans(\delta)\right)\cap P$ over $P$. 
As $\alpha$ and $\beta$ where vertices of $P$, and $\delta = \alpha-\beta$, we must have that $\left(\pi(\alpha)+\spans(\delta)\right)\cap P=\conv(\alpha, \beta)$, and hence $\alpha$, $\beta$ is an edge of $P$. 

The other way around assume that $\alpha$, $\beta$ is an edge of $P$. 
Then there is a cost function $c$ that maximize in $\conv(\alpha, \beta)$ over $P$. 
As $\langle c, \alpha\rangle=\langle c, \beta\rangle$ we have that $c\in W$. 
Therefore, with respect to the basis $b_1,\dots, b_{d-1}, \delta$ we have that $c= (c_1,\dots,c_{d-1}, 0)$ and hence $\langle c, v\rangle=\langle c, \pi(v)\rangle$ for all points $v\in \mathbb{R}^d$ and the result follows.
\end{proof}

That $\pi(\alpha)$ is in the convex hull of $\{\pi(x_i)\}_i$, that is, $\pi(\alpha)$ is not an vertex of the projection of $P$, can be restated as %the following system. 
% \begin{align*}
% \sum_{i} x_i\pi(v_i) &= \pi(\alpha)\\
% \sum_{i} x_i &= 1 \text{ and,}\\
% x_i&\geq 0 \text{ for all }i.
% \end{align*}
% Equivalently, this can be stated as
\begin{align*}
\begin{pmatrix}
\pi(v_1) & \dots & \pi(v_k)\\
1&\dots &1\\
\end{pmatrix}x &= 
\begin{pmatrix}
\pi(\alpha)\\
1
\end{pmatrix}\\
x&\geq \mathbf{0},
\end{align*}
where $v_1,\dots, v_k$ are all vertices of $P$ except $\alpha$ and $\beta$. 
Then adding in a dummy-row $\max_x\ \langle \mathbf{0}, x\rangle$ gives us a linear optimization problem. 
Due to the fact that $\langle \mathbf{0}, x\rangle=0$ for all $x$, this problem has a solution if and only if $\alpha$, $\beta$ is an edge of $P$. 

To solve, and check feasibility of, this type of problems there already exist highly optimized solvers. 
Solving such a problem does however scale worse with dimension than \cref{alg: verify edge num}, however it is still polynomial, with a fixed degree, in $d$. %\marginal{Första gången algoritm 1 nämns. Vad gör den?}
% Therefore this approach seem to scale worse with dimension than \cref{alg: verify edge num}.
Solving the feasibility problem can however verify that $\alpha, \beta$ is not an edge of $P$ and is therefore preferable to only applying \cref{alg: verify edge num}. 
This is especially true since every time we give \cref{alg: verify edge num} a non-edge the algorithm will loop $I$ times, which in practice takes significantly longer than, for example, \texttt{scipy}'s \cite{scipy} \texttt{linprog} \cite{H17} algorithm. 

In \cref{fig: alg scheme} we present the general scheme of how these algorithms can be put together. % \marginal{kanten från "Update L" borde väl gå till romben, inte till DONE?}
In \cref{fig: computation time} we timed our algorithms and compared it to \texttt{polymake}. 
Missing entries indicate that the computations took more than 1 week. 
As is clear, our algorithm is quicker when the dimension becomes large, but is definitely slower for smaller examples. 
The biggest cost however, is verifying all edges, while checking for the rhombus criterion is very fast. 
Therefore, for all polytopes when we have a guarantee that it fulfills the rhombus criterion, the edge-structure can be efficiently computed. 
The code is available at \cite{github-cim}.

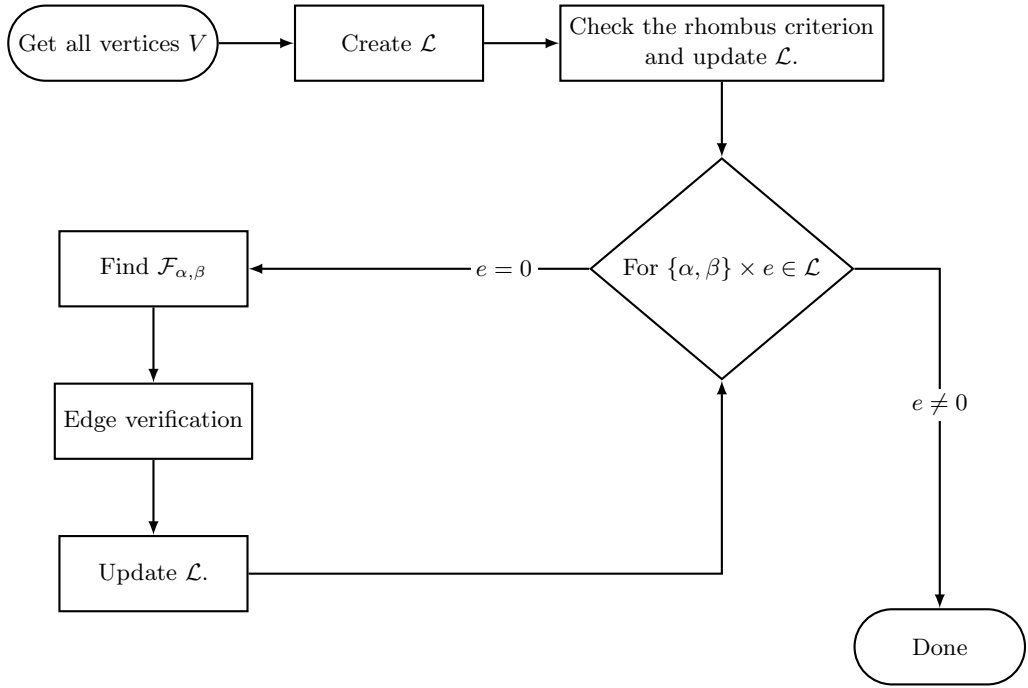
\begin{figure}
\[
\begin{tikzpicture}[font=\small,thick]
% Start block
\node[draw,
    rounded rectangle,
    minimum width=2.5cm,
    minimum height=1cm] (V) {Get all vertices $V$};

\node[draw,
    align=center,
    right=of V,
    minimum width=2.5cm,
    minimum height=1cm
] (L) {Create $\mathcal{L}$};

\node[draw,
    align=center,
    right=of L,
    minimum width=3.5cm,
    minimum height=1cm
] (SQ) {Check the rhombus criterion\\ and update $\mathcal{L}$.};

\node[draw,
    diamond,
    below=of SQ,
    minimum width=3.5cm,
    inner sep=0] (For) {For $\{\alpha,\beta\}\times e\in\mathcal{L}$};
\node[left = of For,
    minimum width=2.5cm,
    minimum height=1cm] (dummy2) {};
\node[draw,
    align=center,
    left =of dummy2,
    minimum width=2.5cm,
    minimum height=1cm
    ] (F) { Find $\mathcal{F}_{\alpha, \beta}$};
\node[draw,
    below=of F,
    minimum width=2.5cm,
    minimum height=1cm] (VER) {Edge verification};
\node[below=of For,
    minimum width=2.5cm,
    minimum height=1cm] (dummy) {};
\node[draw,
    align=center,
    below=of VER,
    minimum width=2.5cm,
    minimum height=1cm] (U) {Update $\mathcal{L}$.};
% Return block
\node[draw,
    rounded rectangle,
    below right=of dummy,
    minimum width=2.5cm,
    minimum height=1cm,] (D) {Done};

% Arrows
\draw[-latex] (V) edge (L)
    (L) edge (SQ)
    (SQ) edge (For);

\draw[-latex] (F.south) -| (VER);
\draw[-latex] (For.east) -| (D) node[pos=0.7,fill=white,inner sep=2pt]{$e\neq 0$};
\draw[-latex] (U) -| (For.south);
\draw[-latex] (For) -- (F) node[pos=0.25,fill=white,inner sep=2pt]{$e=0$};
\draw[-latex] (VER.south) -- (U);
\end{tikzpicture}
\]
\caption{Scheme for computing the edge structure of the discussed polytopes.}
\label{fig: alg scheme}
\end{figure}
\begin{table}
\centering
\begin{tabular}{c | r | r | r | r | r}
Polytope & Polymake & Total & $\mathcal{L}$ creation & Rhombus criterion & Verify edges\\\hline
$\CIM_5$& - & 127 233 & 664 & 259 & 126 310\\
% $B_8$ & - & &&&\\
$\CGP_5$& 623 & 598 & 87.9 & 1.92 & 506\\
$\CGP_4$& 1.2 & 2.68 &1.67&0.02&0.90\\
$B_7$ & 98 343 & 65 484& 256 & 116 & 65 112\\  
$B_6$ & 7.1 & 294 & 60 & 21 & 210 \\
$B_5$ & 1.3 & 6.5 & 1.7  & 0.13 & 4.7 \\
% $P_7$ & 6.7 & 19753 & 659 & 86 & 19008 \\
% $P_6$ & 1.8 & 318 & 11.7 & 0.28 & 306 \\
% $P_5$ & 1.6 & 11.7 & 1.7  & 0.15 & 9.8 \\
\end{tabular}
\caption{Computation time (seconds) for the full charateristic imset polytope, the Chordal graph polytope and the Birkhoff polytope. }
\label{fig: computation time}
\end{table}

Some additional optimizations are possible. 
The current way of verifying edges (\cref{alg: verify edge num} and the feasibility solver) do not require access to the entire list $\mathcal{L}$, only the vertices $V$. 
Therefore, when looping over all $\{\alpha,\beta\}\times e_{\alpha,\beta}\in\mathcal{L}$, parallelization is possible. 
As verifying a single edge is usually very light work, the risk of over-parallelization is significant. 
Therefore doing an additional step after checking for the rhombus criterion, splitting your list into several files (and filtering after only entries with $e=0$) can improve speed, RAM-issues, and make parallelization easier. 
Using these additional optimizations we were able to check that the rhombus criterion holds for $\CGP_6$.

%%%%%%%%%%%%%%%%%%%%%%%%%%
%---SECTION: Acknowledgements
\section{Acknowledgements}
Both authors were partially supported by the Wallenberg AI, Autonomous Systems and Software Program (WASP) funded by the Knut and Alice Wallenberg Foundation.
Svante Linusson was partially supported by Grant (No. 2018-05218) from Vetenskapsr\aa{}det (The Swedish Research Council).

% \begin{table}
% \centering
% \begin{tabular}{c | r | r | r | r | r}
% Polytope & Polymake & Total & $\mathcal{L}$ creation & Rhombus criterion & Verify edges\\\hline
% $\CIM_5$& - & 127 233 & 664 & 259 & 126 310\\
% $B_8$ & - & &&&\\
% $B_7$ & 98 343 & 65 484& 256 & 116 & 65 112\\  
% $B_6$ & 7.1 & 294 & 60 & 21 & 210 \\
% $B_5$ & 1.3 & 6.5 & 1.7  & 0.13 & 4.7 \\
% $P_7$ & 6.7 & 19753 & 659 & 86 & 19008 \\
% $P_6$ & 1.8 & 318 & 11.7 & 0.28 & 306 \\
% $P_5$ & 1.6 & 11.7 & 1.7  & 0.15 & 9.8 \\
% \end{tabular}
% \caption{Computation time (seconds)  for different polytopes. }
% \label{tab: computation time}
% \end{table}

%%%%%%%%%%%%%%%%%%%%%%%%%%
%---BIBLIOGRAPHY
\bibliographystyle{plain}
\bibliography{references}

\begin{thebibliography}{10}

\bibitem{ppl06}
R.~Bagnara, P.~M. Hill, and E.~Zaffanella.
\newblock The {Parma Polyhedra Library}: Toward a complete set of numerical
  abstractions for the analysis and verificatio of hardware and software
  systems.
\newblock Quaderno 457, Dipartimento di Matematica, Universit\`a di Parma,
  Italy, 2006.
\newblock Available at \url{http://www.cs.unipr.it/Publications/}. Also
  published as {\tt arXiv:cs.MS/0612085}, available from
  \url{http://arxiv.org/}.

\bibitem{BRW85}
E.~A. Bender, L.~B. Richmond, and N.~C. Wormald.
\newblock Almost all chordal graphs split.
\newblock {\em Journal of the Australian Mathematical Society},
  38(2):214–221, 1985.

\bibitem{C75}
Va\v{s}ek Chv\'atal.
\newblock On certain polytopes associated with graphs.
\newblock {\em Journal of Combinatorial Theory, Series B}, 18.2:138--154, 1975.

\bibitem{DS18}
Yann Disser and Martin Skutella.
\newblock The simplex algorithm is np-mighty.
\newblock {\em ACM Trans. Algorithms}, 15(1), nov 2018.

\bibitem{EFG16}
Ioannis~Z. Emiris, Vissarion Fisikopoulos, and Bernd Gärtner.
\newblock Efficient edge-skeleton computation for polytopes defined by oracles.
\newblock {\em Journal of Symbolic Computation}, 73:139--152, 2016.

\bibitem{cddlib}
Komei Fukuda.
\newblock \texttt{cddlib}.
\newblock Available at \url{https://github.com/cddlib/cddlib}.

\bibitem{GGPS87}
I.M Gelfand, R.M Goresky, R.D MacPherson, and V.V Serganova.
\newblock Combinatorial geometries, convex polyhedra, and schubert cells.
\newblock {\em Advances in Mathematics}, 63(3):301--316, 1987.

\bibitem{GL09}
Jonna Gill and Svante Linusson.
\newblock The k-assignment polytope.
\newblock {\em Discrete Optimization}, 6(2):148--161, 2009.

\bibitem{H17}
Q.~Huangfu and J.~A.~J. Hall.
\newblock Parallelizing the dual revised simplex method.
\newblock {\em Mathematical Programming Computation}, 10(1):119--142, December
  2017.

\bibitem{LRS22}
Svante Linusson, Petter Restadh, and Liam Solus.
\newblock On the edges of characteristic imset polytopes, 2022.

\bibitem{LRS20}
Svante Linusson, Petter Restadh, and Liam Solus.
\newblock Greedy causal discovery is geometric.
\newblock {\em SIAM Journal on Discrete Mathematics}, 37(1):233--252, 2023.

\bibitem{pandas1}
The pandas~development team.
\newblock pandas-dev/pandas: Pandas, February 2020.

\bibitem{RP87}
G~Rebane and J~Pearl.
\newblock The recovery of causal ploy-trees from statistical data.
\newblock In {\em Proc. of Workshop on Uncertainty in Artificial Intelligence},
  pages 222--228, 1987.

\bibitem{github-cim}
Petter Restadh.
\newblock rhombus-criterion.
\newblock \url{https://github.com/PetterRestadh/rhombus-criterion}, 2023.

\bibitem{S15}
Milan Studen{\'{y}}.
\newblock How matroids occur in the context of learning bayesian network
  structure.
\newblock In Marina Meila and Tom Heskes, editors, {\em Proceedings of the
  Thirty-First Conference on Uncertainty in Artificial Intelligence, {UAI}
  2015, July 12-16, 2015, Amsterdam, The Netherlands}, pages 832--841. {AUAI}
  Press, 2015.

\bibitem{SCK21}
Milan Studen\'{y}, James Cussens, and V\'{a}clav Kratochv\'{\i}l.
\newblock The dual polyhedron to the chordal graph polytope and the rebuttal of
  the chordal graph conjecture.
\newblock {\em Internat. J. Approx. Reason.}, 138:188--203, 2021.

\bibitem{SHL10}
Milan Studen\'y, Raymond Hemmecke, and Silvia Lindner.
\newblock Characteristic imset: A simple algebraic representative of a bayesian
  network structure.
\newblock {\em Proceedings of the 5th European Workshop on Probabilistic
  Graphical Models, PGM 2010}, pages 257--265, 10 2010.

\bibitem{SV08}
Milan Studen\'y and Jiří Vomlel.
\newblock A geometric approach to learning bn structures.
\newblock 01 2008.

\bibitem{CS17}
Milan Studený and James Cussens.
\newblock Towards using the chordal graph polytope in learning decomposable
  models.
\newblock {\em International Journal of Approximate Reasoning}, 88:259 -- 281,
  2017.

\bibitem{scipy}
Pauli Virtanen, Ralf Gommers, Travis~E. Oliphant, Matt Haberland, Tyler Reddy,
  David Cournapeau, Evgeni Burovski, Pearu Peterson, Warren Weckesser, Jonathan
  Bright, St{\'e}fan~J. {van der Walt}, Matthew Brett, Joshua Wilson, K.~Jarrod
  Millman, Nikolay Mayorov, Andrew R.~J. Nelson, Eric Jones, Robert Kern, Eric
  Larson, C~J Carey, {\.I}lhan Polat, Yu~Feng, Eric~W. Moore, Jake
  {VanderPlas}, Denis Laxalde, Josef Perktold, Robert Cimrman, Ian Henriksen,
  E.~A. Quintero, Charles~R. Harris, Anne~M. Archibald, Ant{\^o}nio~H. Ribeiro,
  Fabian Pedregosa, Paul {van Mulbregt}, and {SciPy 1.0 Contributors}.
\newblock {{SciPy} 1.0: Fundamental Algorithms for Scientific Computing in
  Python}.
\newblock {\em Nature Methods}, 17:261--272, 2020.

\bibitem{pandas2}
{W}es {M}c{K}inney.
\newblock {D}ata {S}tructures for {S}tatistical {C}omputing in {P}ython.
\newblock In {S}t\'efan van~der {W}alt and {J}arrod {M}illman, editors, {\em
  {P}roceedings of the 9th {P}ython in {S}cience {C}onference}, pages 56 -- 61,
  2010.

\bibitem{XY12}
Jing Xi and Ruriko Yoshida.
\newblock The characteristic imset polytope of bayesian networks with ordered
  nodes, 2012.

\bibitem{Zie95}
G\"{u}nter~M. Ziegler.
\newblock {\em Lectures on Polytopes}.
\newblock Springer New York, 1995.

\end{thebibliography}

\end{document}